\documentclass{amsart}

\usepackage{amssymb}
\usepackage{amsmath}
\usepackage{amscd}
\usepackage{graphicx}
\usepackage{pifont}
\usepackage{longtable}

\theoremstyle{plain}
\newtheorem{thm}{Theorem}[section]
\newtheorem{cor}[thm]{Corollary}
\newtheorem{lem}[thm]{Lemma}

\newtheorem{prop}[thm]{Proposition}
\theoremstyle{definition}
\newtheorem{defn}[thm]{Definition}
\newtheorem{exmp}[thm]{Example}

\newtheorem{rem}[thm]{Remark}
\newtheorem{prob}[thm]{Problem}
\theoremstyle{remark}

\usepackage{xcolor}
\newcommand{\cb}{}

\begin{document}

\title[Cappell-Shaneson polynomials]
 {Cappell-Shaneson polynomials} 
\author[H.~Endo]{Hisaaki Endo}
\address{Department of Mathematics\\Institute of Science Tokyo\\
2-12-1 Oh-okayama\\Meguro-ku\\Tokyo 152-8551\\Japan}
\email{endo@math.titech.ac.jp}
\author[K.~Iwaki]{Kazunori Iwaki}
\address{Department of Mathematics\\Institute of Science Tokyo\\
2-12-1 Oh-okayama\\Meguro-ku\\Tokyo 152-8551\\Japan}
\email{iwaki.k.ab@m.titech.ac.jp}
\author[A.~Pajitnov]{Andrei Pajitnov}
\address{Laboratoire Math\'ematiques Jean Leray UMR 6629,
Universit\'e de Nantes,
Facult\'e des Sciences,
2, rue de la Houssini\`ere,
44072, Nantes, Cedex} 
\email{andrei.pajitnov@univ-nantes.fr}
\date{September 29, 2024; MSC2020: primary 57K45, secondary 57R40}

\maketitle

\begin{abstract}
{\cb
In the seminal work \cite{CS1976}
}
S. Cappell and J. Shaneson constructed a pair of inequivalent
embeddings of $(n-1)$-spheres in homotopy $(n+1)$-spheres 
for every square matrix of order $n$ with special properties (a Cappell-Shaneson matrix). 
A Cappell-Shaneson polynomial is the characteristic polynomial of a Cappell-Shaneson matrix. 
In this paper, we interpret part of the definition of Cappell-Shaneson polynomial 
as algebraic conditions of polynomials in terms of signed reciprocal polynomial and reduction modulo primes, 
and give complete lists of all Cappell-Shaneson polynomials of degrees $4$ and $5$. 
{\cb We construct several infinite series of
 Cappell-Shaneson polynomials of degrees $6$.}
\end{abstract}


\section{Introduction}


A smooth embedding of $S^{n-1}$ into $S^{n+1}$ is called 
an $(n-1)$-knot (or a knot for simplicity). 
Two knots are said to be equivalent if there exists a self-diffeomorphism of $S^{n+1}$ which maps 
one knot onto the other. 
Since many invariants (such as the Alexander polynomial) of a knot are derived from its complement, 
it is difficult to find inequivalent knots with diffeomorphic complements. 
It is known that there are at most two equivalence classes of knots 
with diffeomorphic complements if $n$ is greater than two (see H. Gluck \cite{Gluck1962},
W. Browder \cite{Browder1967}, M. Kato \cite{Kato1969}, and R. Lashof and J. Shaneson \cite{LS1969}).
S. Cappell and J. Shaneson \cite{CS1976}
constructed first examples of inequivalent knots with diffeomorphic complements. 
Their examples are for $n=4$ and $5$. 
Such examples have been constructed by C. McA. Gordon \cite{Gordon1976} (for $n=3$),
A. Suciu \cite{Suciu1992} (for $n\equiv 4$ and $5 \, ({\rm mod}\, 8)$),
and W. Gu and S. Jiang \cite{GJ1999} (for $n=6$ and $7$).
On the other hand, 
knots which belong to several special classes are known to be determined by their complements 
(cf. \cite[Section 1]{Suciu1992}). 
C. McA. Gordon and J. Luecke \cite{GL1989} proved that every classical knot (i.e. $1$-knot) is determined
by its complement. 

S. Cappell and J. Shaneson \cite{CS1976} constructed a pair of embedded $(n-1)$-spheres $K_0$ and $K_1$
in homotopy $(n+1)$-spheres $\Sigma_0$ and $\Sigma_1$, respectively, 
such that $\Sigma_0-K_0$ is diffeomorphic to $\Sigma_1-K_1$ 
for every $n$ greater than one and every element of $\mathrm{SL}(n,\mathbb{Z})$ with special properties, 
which we call a {\it Cappell-Shaneson matrix} of order $n$. 
A {\it Cappell-Shaneson polynomial} $f(x)$ of degree $n$ is the characteristic polynomial of 
a Cappell-Shaneson matrix $A$ of order $n$. 
The polynomial $f(x)$ is nothing but the Alexander polynomial of $K_0$ and $K_1$ associated with $A$. 
Since the companion matrix of a Cappell-Shaneson polynomial is a Cappell-Shaneson matrix, 
we obtain at least one pair $(K_0,K_1)$ of embedded spheres as above 
once we have a Cappell-Shaneson polynomial. 
If a Cappell-Shaneson polynomial also satisfies a
{\cb
certain
}
positivity condition,
the associated $K_0$ and $K_1$ are inequivalent. 
By classical results on smooth structures on spheres, 
both of $\Sigma_0$ and $\Sigma_1$ are diffeomorphic to $S^{n+1}$ if $n=4$ or $5$. 
Thus a pair of inequivalent knots with diffeomorphic complements is obtained from 
each positive Cappell-Shaneson polynomial of degree $4$ or $5$.

{\cb
During a visit of the third author to Courant Institute in 2019
S. Cappell formulated the following problem:

\begin{quote}
{\it
Do the Cappell-Shaneson matrices exist in every dimension $n\geq 4$?
}
\end{quote}

\noindent
The positive answer to this question would imply the existence of
inequivalent knots with diffeomorphic complements in any dimension $\geq 5$.
The same question was formulated also by D. Ruberman.
}

In this paper we focus our attention on algebraic properties of Cappell-Shaneson polynomials. 
In particular, we give complete lists of all (positive) Cappell-Shaneson polynomials of degrees $4$ and $5$ 
and a complete list of Cappell-Shaneson polynomials of degree $6$ each of which satisfies a 
certain condition on its coefficients, 
interpret part of the definition of Cappell-Shaneson polynomial as algebraic conditions of polynomials 
in terms of signed reciprocal polynomial and reduction modulo primes.

The present paper is organized as follows. 
In Section 2, we give precise definitions of Cappell-Shaneson matrices and polynomials 
and investigate
{\cb
their  properties}.
We introduce a notion of regularity for polynomials with coefficients in a field, 
and interpret the regularity at degree $k$ of a doubly monic polynomial as conditions {\cb on }
its signed reciprocal polynomial and exterior powers when $k$ is equal to $2$ or $3$ 
(see Theorem \ref{2-regular} and Proposition \ref{3-regular}). 
In Section 3, we study the regularity of polynomials with coefficients in $\mathbb{F}_p$ 
and its relation with that of polynomials with integer coefficients. 
In Sections 4 and 5, we give complete list of Cappell-Shaneson polynomials of degrees $4$ and $5$ 
(see Theorems \ref{CS4} and \ref{CS5}). 
In Section 6, we examine Cappell-Shaneson polynomials of degree $6$. 
We give a complete list of Cappell-Shaneson polynomials of degree $6$ 
for which the difference of the coefficients of $x^5$ and $x$ is less than or equal to $12$ 
(see Proposition \ref{CS6} and Appendix \ref{app}). 
In Section 7, we discuss Cappell-Shaneson polynomials of degree greater than or equal to $7$,
{\cb
and state some non-existence results in dimension 8, obtained with the help of SageMath.}


\subsection*{Acknowledgements}\label{ackref}


The first author thanks the Nantes University, the DefiMaths program, and 
``Mission Invite'' for the support and warm hospitality. 
The first author is/was supported by JSPS KAKENHI Grant Numbers 
24K06707, 20K03578, 19H01788. 
The third author was supported by JSPS International Fellowships for Research in Japan 
(Short-term), FY2022 Research Abroad and Invitation Program for International Collaboration, 
and the WRH program. 
{\cb The third author thanks these organisations and Tokyo Institute of Technology
for their support and warm hospitality. The third author is grateful to the University of Miami,
the Courant Institute, and Brandeis University for their support and warm hospitality
during his visits in 2019.}


\section{Cappell-Shaneson polynomials}


In this section we give precise definitions of Cappell-Shaneson matrices and polynomials 
and investigate {\cb their properties}.
See also the paper \cite{CS1976} of S. Cappell and J. Shaneson.
We assume that $n$ is an integer greater than one. 

\subsection{Cappell-Shaneson matrices and polynomials}

We begin with a {\cb  definition } of Cappell-Shaneson matrices.

\begin{defn}\label{CSm}
An element $A$ of $\mathrm{SL}(n,\mathbb{Z})$ is called 
a {\it Cappell-Shaneson matrix} of order $n$ if it satisfies the following condition 
$\mathrm{CS}_k$ for every integer $k$ in $\{1,\ldots ,[n/2]\}$. 

$\mathrm{CS}_k$: the determinant of the matrix $I-\bigwedge^kA$ is equal to $+1$ or $-1$, 
where $I$ is the identity matrix and $\bigwedge^kA$ is the $k$-th exterior power of $A$. 

Let $A$ be a Cappell-Shaneson matrix of order $n$ and $f(x)$ the characteristic polynomial of $A$. 
We say that $A$ is {\it positive} if it satisfies the condition $(-1)^nf(t)>0$ for every $t\in (-\infty,0)$. 
\end{defn}

For every positive Cappell-Shaneson matrix $A$ of order $n$, 
Cappell and Shaneson \cite{CS1976} constructed $(n-1)$-spheres $K_0$ and $K_1$ embedded 
in homotopy $(n+1)$-spheres $\Sigma_0$ and $\Sigma_1$, respectively. 
The Alexander polynomial of each of $K_0$ and $K_1$ 
is equal to the characteristic polynomial of $A$. The exterior of each of $K_0$ and $K_1$ 
admits a fibration with fiber diffeomorphic to the punctured $n$-torus and monodromy $A$. 
Although the exteriors of $K_0$ and $K_1$ are diffeomorphic to each other, 
there is no diffeomorphism from $\Sigma_0$ to $\Sigma_1$ which maps $K_0$ onto $K_1$ 
if $n$ is greater than two. 

\begin{rem}
If $n$ is equal to $2,4$, or $5$, then both of $\Sigma_0$ and $\Sigma_1$ are diffeomorphic to 
$S^{n+1}$. Thus $K_0$ and $K_1$ are not equivalent to each other 
as $(n-1)$-knots in $S^{n+1}$ while they have 
diffeomorphic exteriors if $n$ is equal to $4$ or $5$. 
\end{rem}

We next give the definition of Cappell-Shaneson polynomials. 

\begin{defn}
A monic polynomial $f(x)$ of degree $n$ with integer coefficients is called 
a {\it Cappell-Shaneson polynomial} of degree $n$ if it is the characteristic polynomial of 
a Cappell-Shaneson matrix of order $n$. 
A Cappell-Shaneson polynomial $f(x)$ of degree $n$ is called {\it positive} 
if it satisfies the condition $(-1)^nf(t)>0$ for every $t\in (-\infty,0)$. 
\end{defn}

\begin{exmp}[Cappell-Shaneson polynomials of degree $2$]
Let $A$ be a square matrix of order $2$ with integer entries and $f(x)$ the characteristic polynomial of $A$. 
$A$ belongs to $\mathrm{SL}(2,\mathbb{Z})$ if and only if the constant term of $f(x)$ is equal to $1$. 
$A$ satisfies the condition $\mathrm{CS}_1$ if and only if $f(1)$ is equal to $+1$ or $-1$. 
Hence $f(x)$ is equal to $x^2-x+1$ or $x^2-3x+1$. 
It is easy to see that both of these are positive. 
Since the trace $\mathrm{tr}(A)$ of $A$ must be $1$ or $3$, $A$ is conjugate to one of the 
following matrices: 
\[ 
\left(
\begin{array}{rr}
1 & -1 \\
1 & 0
\end{array}
\right), 
\; 
\left(
\begin{array}{rr}
1 & 1 \\
-1 & 0
\end{array}
\right)
\; 
\textrm{and} 
\;
\left(
\begin{array}{rr}
3 & -1 \\
1 & 0
\end{array}
\right).
\]
These matrices are nothing but the monodromies of the left-handed trefoil, the right-handed trefoil, 
and the figure-eight knot in $S^3$, respectively. 
By virtue of the Gordon-Luecke theorem \cite{GL1989}, 
all $1$-knots obtained from the construction of Cappell and Shaneson \cite{CS1976} are only these three. 
\end{exmp}

\begin{exmp}[Cappell-Shaneson polynomials of degree $3$]
Let $A$ be a square matrix of order $3$ with integer entries and 
$f(x)=x^3+c_2x^2+c_1x+c_0$ the characteristic polynomial of $A$. 
$A$ belongs to $\mathrm{SL}(3,\mathbb{Z})$ if and only if $c_0$ is equal to $-1$. 
$A$ satisfies the condition $\mathrm{CS}_1$ if and only if $f(1)$ is equal to $+1$ or $-1$. 
Hence we have $c_1+c_2=\pm 1$. 
If $c_1+c_2=1$, then we have $f(x)=x^3+c_2x^2+(1-c_2)x-1$ and $f(x)$ is positive if and only if 
$c_2\leq 1$. 
If $c_1+c_2=-1$, then we have $f(x)=x^3+c_2x^2+(-1-c_2)x-1$ and $f(x)$ is positive if and only if 
$c_2\leq 0$. 
We thus obtain all (positive) Cappell-Shaneson polynomials of degree $3$. 
These polynomials play a key role in the study of Cappell-Shaneson homotopy $4$-spheres 
(cf. \cite{AR1984}, \cite{Gompf2010}, \cite{KY2023} and \cite{Iwaki2024}). 
\end{exmp}

We introduce a notion of exterior powers for monic polynomials. 

\begin{lem}\label{ex}
Let $K$ be a field. 
Let $f(x)$ be a monic polynomial of degree $n$ with coefficients in $K$, 
and $A$ a square matrix of order $n$ 
with entries in $K$ whose characteristic polynomial is equal to $f(x)$. 
For an integer $k$ which satisfies $1\leq k\leq n$, 
let $f^{\wedge k}(x)$ denote the characteristic polynomial of $\bigwedge^kA$. 
Then $f^{\wedge k}(x)$ does not depend on a choice of $A$. 
We call $f^{\wedge k}(x)$ the $k$-th {\rm exterior power} of $f(x)$. 
\end{lem}

\begin{proof}
Let $\alpha_1,\ldots,\alpha_n$ be the roots of $f(x)$ in an algebraic closure $\overline{K}$ of $K$. 
The coefficient $s_{\ell}$ of the degree $\ell$ term of $f^{\wedge k}(x)$ 
is equal to the elementary symmetric polynomial of degree $\binom{n}{k}-\ell$ in the variables 
$S=\{\alpha_{i_1}\cdots\alpha_{i_k}\, |\, 1\leq i_1<\cdots <i_k\leq n\}$. 
Since the set $S$ is invariant under the permutations of 
$\alpha_1,\ldots,\alpha_n$, the coefficient $s_{\ell}$ is a symmetric polynomial 
in the variables $\alpha_1,\ldots,\alpha_n$. 
Hence $s_{\ell}$ can be expressed as a polynomial of the elementary symmetric polynomials 
in the variables $\alpha_1,\ldots,\alpha_n$. 
As a consequence, $s_{\ell}$ can be expressed as a polynomial of the coefficients of $f(x)$ 
because the coefficient of the degree $i$ term of $f(x)$ 
is equal to the elementary symmetric polynomial of degree $n-i$ in the variables 
$\alpha_1,\ldots,\alpha_n$. 
Therefore $f^{\wedge k}(x)$ is completely determined by $f(x)$ and $k$. 
\end{proof}

\begin{rem}
The coefficient $s_{\ell}$ of $f^{\wedge k}(x)$ considered in the proof of Lemma \ref{ex} 
can be written in terms of Grothendieck polynomials which play a key role in the theory of 
special $\lambda$-rings. More precisely, it is easily shown that the equality 
$s_{\ell}=(-1)^{N-\ell}P_{N-\ell,k}(c_{n-1},\ldots ,c_0)$ holds, where $N=\binom{n}{k}$, 
$c_i$ is the coefficient of the degree $i$ term of $f(x)$, and $P_{n,m}$ is the Grothendieck polynomial 
(the universal polynomial) defined as in \cite{AT1969}. 
See also \cite{Grothendieck1958}, \cite{Knutson1973}, and 
\cite{Yau2010}. 
\end{rem}

We show that a square matrix which shares the characteristic polynomial with a Cappell-Shaneson 
matrix is also a Cappell-Shaneson matrix. 

\begin{cor}\label{CSCS}
Let $f(x)$ be a monic polynomial of degree $n$ with integer coefficients, 
and $A$ and $B$ square matrices of order $n$ with integer entries. 
Suppose that $f(x)$ is the characteristic polynomial of both of $A$ and $B$. 
For every  integer $k$ which satisfies $1\leq k\leq [n/2]$, 
$A$ satisfies the condition $\mathrm{CS}_k$ if and only if $B$ does. 
Consequently, $A$ is a Cappell-Shaneson matrix if and only if $B$ is. 
\end{cor}

\begin{proof}
The matrix $A$ satisfies the condition $\mathrm{CS}_k$ if and only if the equality
{\cb
$f^{\wedge k}(1)=\pm 1$
}
holds.
The latter is completely determined by $f(x)$ and $k$ by Lemma \ref{ex}. 
\end{proof}

For every Cappell-Shaneson polynomial $f(x)=x^n+c_{n-1}x^{n-1}+\cdots +c_1x+c_0$ of degree $n$, 
the companion matrix 
\[
A
=
\left(
\begin{array}{cccccc}
0 & 1 & 0 & \cdots & \cdots & 0 \\
0 & 0 & 1 & \ddots &  & \vdots \\
\vdots &  & \ddots & \ddots & \ddots & \vdots \\
\vdots &  &  & \ddots & 1 & 0 \\
0 & \cdots & \cdots & \cdots & 0 & 1 \\
-c_0 & -c_1 & -c_2 & \cdots & -c_{n-2} & -c_{n-1}
\end{array}
\right)
\]
of $f(x)$ is a Cappell-Shaneson matrix of order $n$ because of Corollary \ref{CSCS}. 
If $f(x)$ is positive, then $A$ is also positive. 

\subsection{Cappell-Shaneson polynomials and regularity}

We introduce a notion of regularity for polynomials with coefficients in a field. 

\begin{defn}
Let $K$ be a field and $\overline{K}$ an algebraic closure of $K$. 
We consider a monic polynomial $f(x)$ of degree $n$ in $K[x]$ and 
its roots $\alpha_1,\ldots ,\alpha_n$ in $\overline{K}$. 
(A root of $f(x)$ with multiplicity $m$ appears exactly $m$ times in $\alpha_1,\ldots ,\alpha_n$.) 
Let $k$ be an integer which satisfies $1\leq k\leq [n/2]$. 
We say that $f(x)$ is {\it regular} at degree $k$ (or $k$-{\it regular} for short) over $K$ 
if $\alpha_{i_1}\cdots \alpha_{i_k}\ne 1$ for every $k$-tuple $(i_1,\ldots ,i_k)$ of integers 
with $1\leq i_1<\cdots <i_k\leq n$. 
It is clear that $f(x)$ is $1$-regular over $K$ if and only if it satisfies $f(1)\ne 0$. 
We say that $f(x)$ is {\it regular} over $K$ if $f(x)$ is $k$-regular for every integer $k$ 
which satisfies $1\leq k\leq [n/2]$. 
We say that $f(x)$ is {\it doubly monic} if the constant term of $f(x)$ is equal to $(-1)^n$. 
\end{defn}

A square matrix $A$ of order $n$ with integer entries belongs to $\mathrm{SL}(n,\mathbb{Z})$ 
if and only if the characteristic polynomial of $A$ is doubly monic. 

The condition $\mathrm{CS}_k$ on a square matrix of order $n$ 
in Definition \ref{CSm} implies the $k$-regularity of its characteristic polynomial. 

\begin{lem}\label{CS-regular}
Let $k$ be an integer which satisfies $1\leq k\leq [n/2]$. 
If a square matrix $A$ of order $n$ with integer entries satisfies the condition $\mathrm{CS}_k$, 
then the characteristic polynomial $f(x)$ of $A$ is $k$-regular over $\mathbb{Q}$. 
\end{lem}

\begin{proof}
Let $\alpha_1,\ldots ,\alpha_n$ be the roots of $f(x)$ in $\overline{\mathbb{Q}}$. 
The condition $\mathrm{CS}_k$ clearly implies 
that any eigenvalue of $\bigwedge^kA$ is not equal to one. 
{\cb The latter condition  } is equivalent to the $k$-regularity of $f(x)$
because the set of eigenvalues of $\bigwedge^kA$ is equal to 
the set of products $\alpha_{i_1}\cdots \alpha_{i_k}$ 
for all $k$-tuples $(i_1,\ldots ,i_k)$ of integers with $1\leq i_1<\cdots <i_k\leq n$. 
\end{proof}

\begin{rem}
The converse of Lemma \ref{CS-regular} is not true because the value of $\det(I-\bigwedge^kA)$ 
need not be $0$, $1$, or $-1$. Compare with Proposition \ref{CS-regular_p}. 
\end{rem}


We describe a sufficient condition for a polynomial with integer coefficients to be a Cappell-Shaneson 
polynomial. 

\begin{prop}
Let $f(x)$ be an irreducible polynomial of degree $n$ with integer coefficients. 
If the Galois group of $f(x)$ is isomorphic to the symmetric group $S_n$, 
then $f(x)$ is regular over $\mathbb{Q}$. 
\end{prop}

\begin{proof}
Let $\alpha_1,\ldots ,\alpha_n$ be the roots of $f(x)$ in $\overline{\mathbb{Q}}$. 
For each integer $i$ in $\{1,\ldots ,n\}$, the field generated by $\alpha_1,\ldots ,\alpha_i$ 
over $\mathbb{Q}$ is denoted by $K_i$. We obtain the sequence 
$\mathbb{Q}=K_0\subset K_1\subset\cdots\subset K_{n-1}\subset K_n$ of field extensions. 
Since $f(x)$ is irreducible over $\mathbb{Q}$, the degree of the extension 
$K_1/K_0$ is equal to $n$. Let $i$ be an integer in $\{1,\ldots ,n\}$. 
Since $f(x)$ is separable, there exists an element $g(x)$ of $K_{i-1}[x]$ such that 
the equalities $f(x)=(x-\alpha_1)\cdots (x-\alpha_{i-1})g(x)$ and $g(\alpha_i)=0$ hold. 
Hence the degree of the extension $K_i/K_{i-1}$ is less than or equal to 
that of $g(x)$, which is equal to $n-i+1$. Since the degree of the extension $K_n/K_0$ 
is equal to the order of the Galois group of $f(x)$, which is equal to $n!$, we conclude that 
the degree of the extension $K_i/K_{i-1}$ is equal to $n-i+1$. 
In particular, $K_i$ is not equal to $K_{i-1}$ and hence 
$\alpha_1\cdots\alpha_{i-1}\alpha_i\ne 1$. 
The same argument for all permutations of $\alpha_1,\ldots ,\alpha_n$ implies that 
$f(x)$ is $k$-regular over $\mathbb{Q}$ for every integer $k$ in $\{1,\ldots ,[n/2]\}$. 
\end{proof}



\subsection{Signed reciprocal polynomial}

We end this section with the definition and properties of a variation of reciprocal polynomial. 

\begin{defn}
Let $K$ be a field and $f(x)$ a doubly monic polynomial of degree $n$ with coefficients in $K$. 
The doubly monic polynomial $f^*(x)$ defined by $f^*(t)=(-1)^nt^nf(t^{-1})$ is called the 
{\it signed reciprocal polynomial} of $f(x)$. 
If $f(x)$ is the characteristic polynomial of a square matrix $A$ with entries in $K$, 
then $f^*(x)$ is the characteristic polynomial of $A^{-1}$. 
\end{defn}

\begin{prop}\label{reg*}
Let $K$ be a field and $f(x)$ a doubly monic polynomial of degree $n$ with coefficients in $K$. 
Let $k$ be an integer which satisfies $1\leq k\leq [n/2]$. 
Then $f(x)$ is $k$-regular over $K$ if and only if $f^*(x)$ is $k$-regular over $K$. 
\end{prop}

\begin{proof}
Let $\alpha_1,\ldots ,\alpha_n$ be the roots of $f(x)$ in $\overline{K}$. 
Since $f(x)$ is doubly monic, $\alpha_i$ is not equal to $0$ for every $i\in\{1,\ldots ,n\}$. 
By the definition of signed reciprocal polynomial, the roots of $f^*(x)$ is equal to 
$\alpha_1^{-1},\ldots ,\alpha_n^{-1}$. 
For every $k$-tuple $(i_1,\ldots ,i_k)$ of integers with $1\leq i_1<\cdots <i_k\leq n$, the equality 
$\alpha_{i_1}\cdots \alpha_{i_k}=1$ holds if and only if the equality 
$\alpha_{i_1}^{-1}\cdots \alpha_{i_k}^{-1}=1$ holds. Therefore 
$f(x)$ is $k$-regular over $K$ if and only if $f^*(x)$ is $k$-regular over $K$. 
\end{proof}

We now interpret the $k$-regularity of a doubly monic polynomial as conditions of 
its signed reciprocal polynomial and exterior powers when $k$ is equal to $2$ or $3$. 

\begin{thm}\label{2-regular}
Let $K$ be a field and $\overline{K}$ an algebraic closure of $K$. 
Let $f(x)$ be a $1$-regular doubly monic polynomial of degree $n$ with coefficients in $K$. 
If $n$ is greater than $3$, then 
$f(x)$ is $2$-regular over $K$ if and only if there is no polynomial $g(x)$ of degree $2$ 
with coefficients in $\overline{K}$ which divides both of $f(x)$ and $f^*(x)$. 
\end{thm}

\begin{proof}
Let $\alpha_1,\ldots ,\alpha_n$ be the roots of $f(x)$ in $\overline{K}$. 
Since $f(x)$ is doubly monic and $1$-regular over $K$, 
$\alpha_i$ is equal to neither $0$ nor $1$ for every $i\in\{1,\ldots ,n\}$. 
It is easily seen that the roots of $f^*(x)$ are $\alpha_1^{-1},\ldots ,\alpha_n^{-1}$. 

Suppose that $f(x)$ is not $2$-regular over $K$. 
There exist distinct integers $i,j$ in $\{1,\ldots ,n\}$ which satisfy $\alpha_i\alpha_j=1$. 
Both of $\alpha_i=\alpha_j^{-1}$ and $\alpha_j=\alpha_i^{-1}$ are common roots of 
$f(x)$ and $f^*(x)$. Since $i$ is not equal to $j$, $f(x)$ and $f^*(x)$ have the common divisor 
$g(x)=(x-\alpha_i)(x-\alpha_j)$. 

Suppose that $f(x)$ and $f^*(x)$ have a common divisor $g(x)$ of degree $2$ with coefficients 
in $\overline{K}$. There exist elements $a,\alpha,\beta$ of $\overline{K}$ 
which satisfy $g(x)=a(x-\alpha)(x-\beta)$. 
Since both $\alpha$ and $\beta$ are roots of $f(x)$, there exist distinct integers $i,j$ in $\{1,\ldots ,n\}$ 
which satisfy $\alpha_i=\alpha$ and $\alpha_j=\beta$. 
Since both $\alpha$ and $\beta$ are roots of $f^*(x)$, there exist distinct integers $k,\ell$ in $\{1,\ldots ,n\}$ 
which satisfy $\alpha_k^{-1}=\alpha$ and $\alpha_{\ell}^{-1}=\beta$. 
Thus we have $\alpha_i\alpha_k=1$ and $\alpha_j\alpha_{\ell}=1$, 
either of which implies that $f(x)$ is not $2$-regular if $i\ne k$ or $j\ne \ell$. 
If $i=k$ and $j=\ell$, we have $\alpha_i^2=\alpha_j^2=1$, 
and hence $\alpha_i=\alpha_j=-1$ and the characteristic of $K$ is not equal to $2$. 
It also implies that $f(x)$ is not $2$-regular. 
\end{proof}

\begin{prop}\label{3-regular}
Let $K$ be a field and 
$f(x)$ a doubly monic polynomial of degree $n$ with coefficients in $K$. 
If $f(x)$ is separable and $n$ is greater than $5$, 
then $f(x)$ is $3$-regular over $K$ if and only if there is no common root of 
$f^{\wedge 2}(x)$ and $f^*(x)$. 
If $f(x)$ is separable and $n$ is equal to $6$, 
then $f(x)$ is $3$-regular over $K$ if and only if 
$f^{\wedge 2}(x)$ is not
{\cb
divisible }
by $f^*(x)$.
\end{prop}

\begin{proof}
Let $\alpha_1,\ldots,\alpha_n$ be the roots of $f(x)$ in an algebraic closure $\overline{K}$ of $K$. 
Since $f(x)$ is doubly monic, $\alpha_i$ is not equal to $0$ for every $i\in\{1,\ldots ,n\}$. 

We first assume that $n$ is greater than $5$. 

Suppose that $f(x)$ is not $3$-regular over $K$. 
There exist integers $i_1$, $i_2$, $i_3$ which satisfy $1\leq i_1<i_2<i_3\leq n$ 
and $\alpha_{i_1}\alpha_{i_2}\alpha_{i_3}=1$. 
Then $\alpha_{i_1}\alpha_{i_2}=\alpha_{i_3}^{-1}$ is a common root of $f^{\wedge 2}(x)$ and $f^*(x)$. 
Suppose that there exists a common root $\alpha$ of $f^{\wedge 2}(x)$ and $f^*(x)$. 
There exist integers $i_1$, $i_2$, $i_3$ which satisfy $1\leq i_1<i_2\leq n$, $1\leq i_3\leq n$, 
$\alpha=\alpha_{i_1}\alpha_{i_2}$, and $\alpha=\alpha_{i_3}^{-1}$. 
We have $\alpha_{i_1}\alpha_{i_2}\alpha_{i_3}=1$. 
Since $f(x)$ is separable, $\alpha_{i_1}$, $\alpha_{i_2}$, $\alpha_{i_3}$ are distinct and hence 
$i_1$, $i_2$, $i_3$ are also distinct. Therefore $f(x)$ is not $3$-regular over $K$. 

We next assume that $n$ is equal to $6$. 
Suppose that $f(x)$ is not $3$-regular over $K$. 
There exists a permutation $\sigma$ of $\{1,\ldots ,6\}$ which satisfies 
$\alpha_{\sigma(1)}\alpha_{\sigma(2)}\alpha_{\sigma(3)}=1$. 
Since $f(x)$ is doubly monic, we have the equality 
$\alpha_1\alpha_2\alpha_3\alpha_4\alpha_5\alpha_6=1$. 
Hence we also have $\alpha_{\sigma(4)}\alpha_{\sigma(5)}\alpha_{\sigma(6)}=1$. 
Since $f(x)$ is separable, $\alpha_{\sigma(2)}\alpha_{\sigma(3)}=\alpha_{\sigma(1)}^{-1}$, 
$\alpha_{\sigma(3)}\alpha_{\sigma(1)}=\alpha_{\sigma(2)}^{-1}$, 
$\alpha_{\sigma(1)}\alpha_{\sigma(2)}=\alpha_{\sigma(3)}^{-1}$, 
$\alpha_{\sigma(5)}\alpha_{\sigma(6)}=\alpha_{\sigma(4)}^{-1}$, 
$\alpha_{\sigma(6)}\alpha_{\sigma(4)}=\alpha_{\sigma(5)}^{-1}$, 
$\alpha_{\sigma(4)}\alpha_{\sigma(6)}=\alpha_{\sigma(6)}^{-1}$ are distinct common roots of 
$f^{\wedge 2}(x)$ and $f^*(x)$. It implies that $f^{\wedge 2}(x)$ is divided by $f^*(x)$. 
Suppose that $f^{\wedge 2}(x)$ is divided by $f^*(x)$. 
There exists a common root of $f^{\wedge 2}(x)$ and $f^*(x)$ in $\overline{K}$. 
By the same argument as above, $f(x)$ is not $3$-regular over $K$. 
\end{proof}


\section{Reduction modulo primes}


Let $p$ be a prime number and $\mathbb{F}_p$ the prime field of order $p$. 
In this section we study the regularity of polynomials with coefficients in $\mathbb{F}_p$ 
and its relation with that of polynomials with integer coefficients. 
We assume that $n$ is an integer greater than one. 

\subsection{Regularity over $\boldsymbol{\mathbb{F}_p}$}

We first describe a sufficient condition for a polynomial with coefficients in $\mathbb{F}_p$ to be regular. 

\begin{prop}\label{reg_p}
Let $p$ be a prime number and $f(x)$ a polynomial of degree $n$ with coefficients in $\mathbb{F}_p$. 
If $f(x)$ is irreducible
{\cb and primitive},
then $f(x)$ is regular over $\mathbb{F}_p$.
\end{prop}

\begin{proof}
Let $\overline{\mathbb{F}}_p$ be an algebraic closure of $\mathbb{F}_p$ and 
$\alpha$ a root of $f(x)$ in $\overline{\mathbb{F}}_p$. 
The field $\mathbb{F}_p(\alpha)$ generated by $\alpha$ over $\mathbb{F}_p$ is 
an extension field of $\mathbb{F}_p$ of degree $n$ because $f(x)$ is irreducible over $\mathbb{F}_p$. 
Hence $\mathbb{F}_p(\alpha)$ is equal to 
$\mathbb{F}_q=\{t\in\overline{\mathbb{F}}_p\, |\, t^q=t\}\; (q=p^n)$. 

By a property of the Frobenius endomorphism, we have $f(\alpha^p)=0$. 
Similarly, if we assume that $f(\alpha^{p^{i-1}})=0$, then we have $f(\alpha^{p^{i}})=0$ 
for every integer $i$ in $\{1,\ldots ,n-1\}$. Since the multiplicative group of $\mathbb{F}_q$ is 
a cyclic group of order $q-1$, we conclude that the set of roots of $f(x)$ is equal to 
$R=\{\alpha^{p^i}\, |\, i=0,\ldots ,n-1\}$. 

For an integer $k$ in $\{1,\ldots ,[n/2]\}$ and integers 
$i_1,\ldots ,i_k$ with $0\leq i_1<\cdots <i_k\leq n-1$, 
we consider the sum $s=s(i_1,\ldots ,i_k)=p^{i_1}+\cdots +p^{i_k}$. Then we obtain 
\[
s\leq p^{n-[n/2]}+\cdots +p^{n-1}<1+p+p^2+\cdots +p^{n-1}=\frac{p^n-1}{p-1}\leq q-1.
\]
Every product of $k$ elements of $R$ is expressed as $\alpha^s$ for some $k$-tuple 
$(i_1,\ldots ,i_k)$ of integers with $0\leq i_1<\cdots <i_k\leq n-1$. 
Since the multiplicative group of $\mathbb{F}_q$ is a cyclic group of order $q-1$, 
$\alpha^s$ is not equal to $1$
{\cb since it is a generator of the multiplicative group  $(\mathbb{F}_p(\alpha))^*$}.
Therefore $f(x)$ is $k$-regular over $\mathbb{F}_p$. 
\end{proof}

\begin{rem}
The irreducibility over $\mathbb{F}_p$ assumed in Proposition \ref{reg_p} is not a necessary 
condition for a polynomial to be regular if $p\geq 3$. For example, the polynomial 
\[
f(x)=x^8+x^7-x^6+x^5+x+1=(x^4-x^3-x^2+x-1)^2
\]
with coefficients in $\mathbb{F}_3$ is reducible, while it is regular over $\mathbb{F}_3$. 
\end{rem}

On the other hand, 
the irreducibility over $\mathbb{F}_2$ is a necessary condition for a polynomial to be regular. 

\begin{prop}\label{irred}
Let $f(x)$ be a polynomial of degree $n$ with coefficients in $\mathbb{F}_2$. 
If $f(x)$ is regular over $\mathbb{F}_2$ and its constant term is not equal to zero, 
then $f(x)$ is irreducible over $\mathbb{F}_2$. 
\end{prop}

\begin{proof}
Suppose that $f(x)$ is reducible over $\mathbb{F}_2$. 
There exist polynomials $g(x),\, h(x)$ of positive degrees with coefficients in $\mathbb{F}_2$ 
which satisfy $f(x)=g(x)h(x)$. We can assume without loss of generality 
that the degree $k$ of $g(x)$ is less than or equal to $[n/2]$. 
Since the constant term of $f(x)$ is equal to $1$, that of $g(x)$ must be equal to $1$. 
Hence the product of all $k$ roots of $g(x)$ is equal to $1$, 
which implies that $f(x)$ is not $k$-regular over $\mathbb{F}_2$. 
\end{proof}

\newcommand{\ff}{\mathbb F}
\newcommand{\nn}{\mathbb N}
\newcommand{\zz}{\mathbb Z}

{\cb
\begin{prop}\label{t:cs-primitive}
A polynomial $P\in\ff_2[x]$ with a
non-zero free term
is regular
f and only if
it is irreducible and primitive.
\end{prop}
\begin{proof}
Let $P$ be a regular polynomial.
Let
$\l\in\ff_{2^n}$
be any root of $P$, then the sequence
of all roots of $P$ is of the form
$$\l_0=\l, \ \ \l_1=\l^{2},\ \  \cdots,
\l_{n-1}=\l^{2^{n-1}}.$$
Let $r\in\nn, r< 2^{n}-1$.
There is a unique sequence
$a_0,\ldots,a_{n-1}$ with $a_i\in\{0,1\}$
such that
$r=\sum_ia_i\cdot 2^i$
(the dyadic expansion of $r$).
At least one of coefficients $a_i$
equals zero, since $r< 2^n-1$.
Observe that
$$
\l^r=\prod_{a_i\not=0} \l_i=
\prod_{a_i\not=0} \l^{2^i}.
$$
This is a product of pairwise different roots of $P$,
containing at most
$n-1$ roots. The condition
$CS$ implies that this product is not equal to $1$,
therefore $\l^r\not=1$, and the order of $\l$ in the
group
$\ff_{2^n}^*$
equals indeed
$2^n -1$.
\end{proof}
}

\subsection{Reduction modulo primes}

We next consider reductions of integer polynomials modulo prime numbers. 
{\cb
\begin{defn}
For  a polynomial $f(x)\in \zz[x]$ we denote by
$f_p(x)\in\ff_p[x]$
its reduction $mod\ p$.
Similarly for a matrix $A$ with integer coefficients
we denote by  $A_p$ its reduction $mod\ p$.
\end{defn}
}


\begin{prop}\label{CS-regular_p}
Let $k$ be an integer which satisfies $1\leq k\leq [n/2]$. 
Let $A$ be a square matrix of order $n$ with integer entries and $f(x)$ its characteristic polynomial. 
Then $A$ satisfies the condition $\mathrm{CS}_k$ 
if and only if $f_p(x)$ is $k$-regular over $\mathbb{F}_p$ for every prime number $p$. 
\end{prop}

\begin{proof}
It is not difficult to see that $A$ satisfies the condition $\mathrm{CS}_k$ if and only if 
the integer $\det (I-\bigwedge^kA)$ is not divisible by any prime number $p$. 
Further, the latter is equivalent to the condition that $\det(I-\bigwedge^kA_p)\ne 0$ holds in $\mathbb{F}_p$ 
for every prime number $p$. Since $f_p(x)$ is the characteristic polynomial of $A_p$, 
this is equivalent to the condition that $f_p(x)$ is $k$-regular over $\mathbb{F}_p$ for every prime number $p$ 
because the set of eigenvalues of $\bigwedge^kA_p$ is equal to 
the set of products $\alpha_{i_1}\cdots \alpha_{i_k}$ 
for all $k$-tuples $(i_1,\ldots ,i_k)$ of integers with $1\leq i_1<\cdots <i_k\leq n$, 
where $\alpha_1,\ldots ,\alpha_n$ is the roots of $f(x)$ in an algebraic closure $\overline{\mathbb{F}}_p$ 
of $\mathbb{F}_p$. 
\end{proof}



The next proposition was first proved by Gu and Jiang \cite[Theorem 3.2]{GJ1999};
{\cb
we suggest another proof.

\begin{prop}
Every Cappell-Shaneson polynomial is irreducible over $\zz$.
\end{prop}
}
\begin{proof}
Let $f(x)$ be a Cappell-Shaneson polynomial and $A$ the companion matrix of $f(x)$. 
Since $A$ is a Cappell-Shaneson matrix by Lemma \ref{CSCS}, 
$f_2(x)$ is regular over $\mathbb{F}_2$ by Proposition \ref{CS-regular_p}. 
Since $f(x)$ is doubly monic, the constant term of $f_2(x)$ is equal to $1$. 
Hence $f_2(x)$ is irreducible over $\mathbb{F}_2$ by Proposition \ref{irred}. 
Therefore $f(x)$ is irreducible over
{\cb
$\mathbb{Z}$.}
\end{proof}

\begin{rem}
There exists a doubly monic polynomial of degree $n$ with integer coefficients whose Galois group 
is isomorphic to $S_n$ and whose reduction modulo $2$ is not regular over $\mathbb{F}_2$. 
For example, the Galois group of $f(x)=x^5-x-1$ is isomorphic to $S_5$, while 
$f_2(x)=(x^2+x+1)(x^3+x^2+1)$ is reducible over $\mathbb{F}_2$, and hence $f_2(x)$ is not regular 
over $\mathbb{F}_2$ by Proposition \ref{irred}. 
\end{rem}

\begin{prop}\label{CSdual}
A doubly monic polynomial $f(x)$ of degree $n$ with integer coefficients is 
a Cappell-Shaneson polynomial if and only if $f^*(x)$ is a Cappell-Shaneson polynomial. 
\end{prop}

\begin{proof}
Straightforward from Propositions \ref{reg*} and \ref{CS-regular_p}. 
\end{proof}

\begin{exmp}[Doubly monic $1$- and $2$-regular polynomials] 
For every even integer $n$ greater than one, we show that the doubly monic polynomial 
$f(x)=x^{2n+2}-x^n+1$ is $1$-regular and $2$-regular over $\mathbb{Q}$. 
Since $f(1)=1\ne 0$, every root of $f(x)$ is not equal to $1$. Hence $f(x)$ is $1$-regular over $\mathbb{Q}$. 
Assume that $f(x)$ is not $2$-regular over $\mathbb{Q}$. 
By Lemma \ref{CS-regular}, the companion matrix $A$ of $f(x)$ does not satisfy the condition $\mathrm{CS}_2$. 
Hence $f_p(X)$ is not $2$-regular over $\mathbb{F}_p$ for some prime number $p$ 
by Proposition \ref{CS-regular_p}. 
It follows from Theorem \ref{2-regular} that there exists a polynomial $g(x)$ of 
degree $2$ with coefficients in $\overline{\mathbb{F}}_p$ which divides 
both of $f_p(x)$ and $f_p^*(x)$. Since we have
\[
f_p(x)-f_p^*(x)=x^n(x-1)(x+1), 
\]
$g(x)$ must be $x^2$, $x(x-1)$, $x(x+1)$, or $(x-1)(x+1)$. 
Then at least one of $f_p(0)$, $f_p(1)$, $f_p(-1)$ is equal to $0$ because $f_p(x)$ is divisible by 
$g(x)$. On the other hand, we have $f_p(0)=f_p(1)=f_p(-1)=1$ because $n$ is even. 
This is a contradiction. Hence $f(x)$ is $2$-regular over $\mathbb{Q}$. 
\end{exmp}

\begin{exmp}[A doubly monic $3$-regular polynomial of degree $6$] 
We show that the doubly monic polynomial 
$f(x)=x^6+x^5-x^4-2x^3+x+1$ is $3$-regular over $\mathbb{Q}$. 
The formal derivative $f'(x)$ of $f(x)$ is $6x^5+5x^4-4x^3-6x^2+1$. 
It is not difficult to see that $f_p(x)$ and $f'_p(x)$ are coprime and hence $f_p(x)$ is separable 
for every prime number $p$. 
By direct computation, we obtain 
{\allowdisplaybreaks %
\begin{align*}
f^{\wedge 2}(x) & =x^{15}+x^{14}-2x^{13}-4x^{12}-x^{11}+3x^{10}+3x^9+2x^8 \\
& -x^7-4x^6-x^5+x^4+3x^3+x^2-1 
\end{align*}}
and $f^*(x)=x^6+x^5-2x^3-x^2+x+1$. 
The remainder in the division of $f^{\wedge 2}(x)$ by $f^*(x)$ is equal to 
$r(x)=4x^5-7x^4+7x^2+2x-4$. 
Since the greatest common divisor of the coefficients of $r(x)$ is equal to $1$, 
$f_p^{\wedge 2}(x)$ is not divided by $f_p^*(x)$ for every prime number $p$. 
By Proposition \ref{3-regular}, $f_p(x)$ is $3$-regular over $\mathbb{F}_p$ for every prime number $p$, 
which implies that the companion matrix $A$ of $f(x)$ satisfies the condition $\mathrm{CS}_3$. 
Therefore $f(x)$ is $3$-regular over $\mathbb{Q}$ by Lemma \ref{CS-regular}. 
\end{exmp}


\section{Degree 4: three approaches}


In this section we give a complete list of Cappell-Shaneson polynomials of degree $4$. 
We first state the main theorem of this section. 

\begin{thm}\label{CS4}
Let $f(x)=x^4+c_3x^3+c_2x^2+c_1x+c_0$ be a monic polynomial of degree $4$ with integer coefficients. 
Then $f(x)$ is a Cappell-Shaneson polynomial if and only if the $4$-tuple $(c_0,c_1,c_2,c_3)$ of 
its coefficients is equal to one of those exhibited below, where $a$ is an integer. 
Further, $f(x)$ is positive if and only if its coefficients satisfy 
the positivity condition indicated in each row in the table below. 

\medskip

\begin{center}
\begin{tabular}{|c|c|c|c|c|} \hline 
$c_0$ & $c_1$ & $c_2$ & $c_3$ & {\rm positivity} \\ \hline 
$1$ & $a-1$ & $-2a$ & $a$ & $a\leq 0$ \\ \hline
$1$ & $a-1$ & $-2a-2$ & $a$ & $a\leq 0$ \\ \hline
\end{tabular} 
\begin{tabular}{|c|c|c|c|c|} \hline 
$c_0$ & $c_1$ & $c_2$ & $c_3$ & {\rm positivity} \\ \hline 
$1$ & $a+1$ & $-2a-2$ & $a$ & $a\leq -1$ \\ \hline
$1$ & $a+1$ & $-2a-4$ & $a$ & $a\leq -1$ \\ \hline
\end{tabular} 
\end{center}

\end{thm}

\begin{rem}
Two families of polynomials in the first row in the table above are `reciprocal': 
one family consists of the (signed) reciprocal polynomials of all polynomials in the other family. 
Two families of polynomials in the second row in the table above are also `reciprocal' 
(see Proposition \ref{CSdual}). 
\end{rem}

\begin{rem}
The polynomials $f(x)=x^4+ax^3-2(a+1)x^2+(a+1)x+1\, (a\leq -1)$ were found by Cappell and Shaneson 
\cite{CS1976}. The polynomials $f(x)=x^4+ax^3-2ax^2+(a-1)x+1\, (a\leq 0)$ were found by 
Gu and Jiang \cite{GJ1999}. 
\end{rem}

We will describe three different proofs of the first part of Theorem \ref{CS4} in the following subsections. 
It is not difficult to check the positivity conditions. 

\subsection{The first proof: companion matrix}

Let $A$ be the companion matrix of $f(x)$. 
Since $f(x)$ is the characteristic polynomial of $A$, $f(x)$ is a Cappell-Shaneson polynomial 
if and only if $A$ is a Cappell-Shaneson matrix by Corollary \ref{CSCS}. 
The condition $\det A=1$ is equivalent to the condition $c_0=1$. 
Since we have $\det(I-A)=f(1)=c_1+c_2+c_3+2$, 
$A$ satisfies the condition $\mathrm{CS}_1$ if and only if $c_1+c_2+c_3=-1,-3$. 
By direct computation of $\bigwedge^2A$, we obtain the equality $\det(I-\bigwedge^2A)=-(c_3-c_1)^2$. 
Hence $A$ satisfies the condition $\mathrm{CS}_2$ if and only if the equality $c_3-c_1=\pm 1$ holds. 
Therefore we obtain four families of Cappell-Shaneson polynomials shown above. 

\subsection{The second proof: symmetric polynomials}

Let $A$ be a Cappell-Shaneson matrix of order $4$ with characteristic polynomial $f(x)$. 
The condition $\det A=1$ is equivalent to the condition $c_0=1$. 
Since we have $\det(I-A)=f(1)=c_1+c_2+c_3+2$, 
$A$ satisfies the condition $\mathrm{CS}_1$ if and only if $c_1+c_2+c_3=-1,-3$. 

Let $\alpha_1,\alpha_2,\alpha_3,\alpha_4$ be the roots of $f(x)$ in $\overline{\mathbb{Q}}$. 
The coefficient $s_k$ of the degree $k$ term of the characteristic polynomial 
$f^{\wedge 2}(x)$ of $\bigwedge^2A$ 
is equal to the elementary symmetric polynomial of degree $6-k$ in the valuables 
$S=\{\alpha_i\alpha_j\, |\, 1\leq i<j\leq 4\}$. Since the set $S$ is invariant under the permutations of 
$\alpha_1,\alpha_2,\alpha_3,\alpha_4$, the coefficient $s_k$ is a symmetric polynomial 
in the valuables $\alpha_1,\alpha_2,\alpha_3,\alpha_4$. 
Hence $s_k$ can be expressed as a polynomial of the elementary symmetric polynomials 
in the valuables $\alpha_1,\alpha_2$, $\alpha_3,\alpha_4$. 
As a consequence, $s_k$ can be expressed as a polynomial of $c_1,c_2,c_3$ 
because $c_i$ is equal to the elementary symmetric polynomial of degree $4-i$ in the valuables 
$\alpha_1,\alpha_2,\alpha_3,\alpha_4$ (see also Lemma \ref{ex}). In fact, we have 
\[
s_1=s_5=-c_2,\quad s_2=s_4=c_1c_3-1,\quad s_3=2c_2-c_1^2-c_3^2
\]
and thus $f^{\wedge 2}(x)=x^6-c_2x^5+(c_1c_3-1)x^4+(2c_2-c_1^2-c_3^2)x^3+(c_1c_3-1)x^2-c_2x+1$. 
We obtain the equality $\det(I-\bigwedge^2A)=f^{\wedge 2}(1)=-(c_3-c_1)^2$. 
Hence $A$ satisfies the condition $\mathrm{CS}_2$ if and only if the equality $c_3-c_1=\pm 1$ holds. 

Therefore we obtain four families of Cappell-Shaneson polynomials shown above. 

\subsection{The third proof: signed reciprocal polynomial}

Let $A$ be a Cappell-Shaneson matrix of order $4$ with characteristic polynomial $f(x)$. 
The condition $\det A=1$ is equivalent to the condition $c_0=1$. 
Since we have $\det(I-A)=f(1)=c_1+c_2+c_3+2$, 
$A$ satisfies the condition $\mathrm{CS}_1$ if and only if $c_1+c_2+c_3=-1,-3$. 

Since the signed reciprocal polynomial $f^*(x)$ of $f(x)$ is equal to $x^4+c_1x^3+c_2x^2+c_3x+1$, 
we have 
\[
f(x)-f^*(x)=(c_3-c_1)x(x-1)(x+1). 
\]
If $c_3-c_1$ is not equal to $\pm 1$, it is divisible by some prime number $p$. 
From the equality above, $f_p(x)$ is equal to $f_p^*(x)$ as an element of $\overline{\mathbb{F}}_p[x]$. 
There exists a polynomial of degree $2$ in $\overline{\mathbb{F}}_p[x]$ such that it divides both of $f_p(x)$ and 
$f_p^*(x)$, which implies that $f_p(x)$ is not $2$-regular over $\mathbb{F}_p$ by 
Theorem \ref{2-regular}. Hence $A$ does not satisfy the condition $\mathrm{CS}_2$ 
by Proposition \ref{CS-regular_p}. 
If $c_3-c_1$ is equal to $\pm 1$, we have $f_p(x)-f_p^*(x)=\pm x(x-1)(x+1)$ 
for every prime number $p$. 
Both of $f_p(x)$ and $f_p^*(x)$ are not divisible by $x$ because of $f_p(0)=f_p^*(0)=1$. 
Both of $f_p(x)$ and $f_p^*(x)$ are not divisible by $x-1$ because of $f_p(1)=f_p^*(1)=\pm 1$. 
Therefore the degree of a common divisor of $f_p(x)$ and $f_p^*(x)$ is less than $2$, 
which means that $f_p(x)$ is $2$-regular over $\mathbb{F}_p$ for every prime number $p$ 
by Theorem \ref{2-regular}. From Proposition \ref{CS-regular_p}, 
$A$ satisfies the condition $\mathrm{CS}_2$ if and only if the equality $c_3-c_1=\pm 1$ holds. 

Therefore we obtain four families of Cappell-Shaneson polynomials shown above.


\section{Degree 5}


In this section we give a complete list of Cappell-Shaneson polynomials of degree $5$. 
We state the main theorem of this section. 

\begin{thm}\label{CS5}
Let $f(x)=x^5+c_4x^4+c_3x^3+c_2x^2+c_1x+c_0$ be 
a monic polynomial of degree $5$ with integer coefficients. 
Then $f(x)$ is a Cappell-Shaneson polynomial if and only if the $5$-tuple $(c_0,c_1,c_2,c_3,c_4)$ of 
its coefficients is equal to one of those exhibited below, where $a$ and $b$ are integers. 


\begin{longtable}{|c|c|c|c|c|c|} \hline 
{\rm Case} & $c_0$ & $c_1$ & $c_2$ & $c_3$ & $c_4$ \\ \hline 
{\rm I-i-1} & $-1$ & $-a+1$ & $2a+1$ & $-2a-1$ & $a$ \\ \hline
{\rm I-i-2} & $-1$ & $-2a+2$ & $5a$ & $-5a-1$ & $2a$ \\ \hline
{\rm I-i-3} & $-1$ & $-a-1$ & $4a+15$ & $-4a-13$ & $a$ \\ \hline
{\rm I-i-4} & $-1$ & $-2a-2$ & $7a+13$ & $-7a-10$ & $2a$ \\ \hline
{\rm I-ii-1} & $-1$ & $-a$ & $-b+1$ & $b$ & $a$ \\ \hline
{\rm I-ii-2} & $-1$ & $(-b+1)a+5$ & $(3b-2)a-b-9$ & $(-3b+1)a+b+10$ & $ab-5$ \\ \hline
{\rm II-i-1} & $-1$ & $-a+1$ & $4a+9$ & $-4a-11$ & $a$ \\ \hline
{\rm II-i-2} & $-1$ & $-2a+2$ & $7a+3$ & $-7a-6$ & $2a$ \\ \hline
{\rm II-i-3} & $-1$ & $-a-1$ & $2a+3$ & $-2a-3$ & $a$ \\ \hline
{\rm II-i-4} & $-1$ & $-2a-2$ & $5a+6$ & $-5a-5$ & $2a$ \\ \hline
{\rm II-ii-1} & $-1$ & $-a$ & $-b-1$ & $b$ & $a$ \\ \hline
{\rm II-ii-2} & $-1$ & $(-b+1)a+5$ & $(3b-2)a+b-11$ & $(-3b+1)a-b+10$ & $ab-5$ \\ \hline
\end{longtable} 

\end{thm}

The method of the proof of Theorem \ref{CS5} is basically similar to that of 
the first proof of Theorem \ref{CS4}, while the proof itself is much more complicated. 

\begin{proof}
Let $A$ be the companion matrix of $f(x)$. 
Since $f(x)$ is the characteristic polynomial of $A$, $f(x)$ is a Cappell-Shaneson polynomial 
if and only if $A$ is a Cappell-Shaneson matrix by Corollary \ref{CSCS}. 
The condition $\det A=1$ is equivalent to the condition $c_0=-1$. 
Since we have $\det(I-A)=f(1)=c_1+c_2+c_3+c_4$, 
$A$ satisfies the condition $\mathrm{CS}_1$ if and only if the equality 
\begin{equation}
c_1+c_2+c_3+c_4=\pm 1 \tag{A}
\end{equation}
holds. 
By direct computation of $\det(I-\bigwedge^2A)$, 
we know that $A$ satisfies the condition $\mathrm{CS}_2$ if and only if the equality 
\begin{equation}
\begin{split}
& 3c_2c_4-c_2^2-c_1c_4^2+3c_1c_3+c_1c_2c_4-c_1^2-c_4^3+3c_3c_4
-c_1^2c_3+c_2c_4^2 \\
& -2c_2c_3-2c_1c_4+c_1^3-c_1c_3c_4+3c_1c_2
-c_4^2+c_1^2c_4-c_3^2=\pm 1 
\end{split}\tag{B}
\end{equation}
holds. 

Case I: 
Suppose that the right-hand side of (A) is equal to $1$, that is, $c_1+c_2+c_3+c_4=1$. 
Substituting $1-c_1-c_3-c_4$ for $c_2$, we can show that the left-hand side of (B) 
is equal to $uv-1$, 
where
\[
u=c_1+c_4,\quad 
v=(c_1-c_3-2c_4-5)u+c_4+5. 
\]

Case I-i: 
Suppose that the right-hand side of (B) is equal to $1$. 
Then the equation (B) is equivalent to $uv=2$. 
Since $u$ and $v$ are integers, the pair $(u,v)$ is equal to $(1,2), (2,1), (-1,-2)$, or $(-2,-1)$. 

If $(u,v)=(1,2)$, then we have $c_1=1-c_4$, $c_3=-2c_4-1$ and $c_2=2c_4+1$. 
Hence we obtain 
$(c_1,c_2,c_3,c_4)=(-a+1,2a+1,-2a-1,a)$
for some integer $a$. 

If $(u,v)=(2,1)$, then we have $c_1=2-c_4$, $c_3=(-5c_4-2)/2$ and $c_2=5c_4/2$. 
Hence we obtain 
$(c_1,c_2,c_3,c_4)=(-2a+2,5a,-5a-1,2a)$ 
for some integer $a$. 

If $(u,v)=(-1,-2)$, then we have $c_1=-1-c_4$, $c_3=-4c_4-13$ and $c_2=4c_4+15$. 
Hence we obtain 
$(c_1,c_2,c_3,c_4)=(-a-1,4a+15,-4a-13,a)$ 
for some integer $a$. 

If $(u,v)=(-2,-1)$, then we have $c_1=-2-c_4$, $c_3=(-7c_4-20)/2$ and $c_2=(7c_4+26)/2$. 
Hence we obtain 
$(c_1,c_2,c_3,c_4)=(-2a-2,7a+13,-7a-10,2a)$ 
for some integer $a$. 

Case I-ii: 
Suppose that the right-hand side of (B) is equal to $-1$. 
Then the equation (B) is equivalent to $uv=0$, which implies that $u=0$ or $v=0$. 

If $u=0$, then we have $c_1=-c_4$ and $c_2=1-c_3$. 
Hence we obtain 
$(c_1,c_2,c_3,c_4)=(-a,-b+1,b,a)$ 
for some integers $a$ and $b$. 

If $u\ne 0$ and $v=0$, then $c_4+5$ is divided by $u$ 
because we have $c_4+5=-(c_1-c_3-2c_4-5)u$. 
We rewrite $u$ and $c_4+5$ as $a$ and $ab$, respectively, 
where $a$ is a non-zero integer and $b$ is an integer. 
Then we obtain 
$(c_1,c_2,c_3,c_4)=((-b+1)a+5, (3b-2)a-b-9, (-3b+1)a+b+10, ab-5)$. 

Case II: 
Suppose that the right-hand side of (A) is equal to $-1$, that is $c_1+c_2+c_3+c_4=-1$. 
Substituting $-1-c_1-c_3-c_4$ for $c_2$, we can show that the left-hand side of (B) 
is equal to $uv-1$, 
where
\[
u=c_1+c_4,\quad 
v=(c_1-c_3-2c_4-5)u-c_4-5. 
\]

Case II-i: 
Suppose that the right-hand side of (B) is equal to $1$. 
Then the equation (B) is equivalent to $uv=2$. 
Since $u$ and $v$ are integers, the pair $(u,v)$ is equal to $(1,2), (2,1), (-1,-2)$, or $(-2,-1)$. 

If $(u,v)=(1,2)$, then we have $c_1=1-c_4$, $c_3=-4c_4-11$ and $c_2=4c_4+9$. 
Hence we obtain 
$(c_1,c_2,c_3,c_4)=(-a+1,4a+9,-4a-11,a)$ 
for some integer $a$. 

If $(u,v)=(2,1)$, then we have $c_1=2-c_4$, $c_3=(-7c_4-12)/2$ and $c_2=(7c_4+6)/2$. 
Hence we obtain 
$(c_1,c_2,c_3,c_4)=(-2a+2,7a+3,-7a-6,2a)$ 
for some integer $a$. 

If $(u,v)=(-1,-2)$, then we have $c_1=-1-c_4$, $c_3=-2c_4-3$ and $c_2=2c_4+3$. 
Hence we obtain 
$(c_1,c_2,c_3,c_4)=(-a-1,2a+3,-2a-3,a)$ 
for some integer $a$. 

If $(u,v)=(-2,-1)$, then we have $c_1=-2-c_4$, $c_3=(-5c_4-10)/2$ and $c_2=(5c_4+12)/2$. 
Hence we obtain 
$(c_1,c_2,c_3,c_4)=(-2a-2,5a+6,-5a-5,2a)$ 
for some integer $a$. 

Case II-ii: 
Suppose that the right-hand side of (B) is equal to $-1$. 
Then the equation (B) is equivalent to $uv=0$, which implies that $u=0$ or $v=0$. 

If $u=0$, then we have $c_1=-c_4$ and $c_2=-1-c_3$. 
Hence we obtain 
$(c_1,c_2,c_3,c_4)=(-a,-b-1,b,a)$ 
for some integers $a$ and $b$. 

If $u\ne 0$ and $v=0$, then $c_4+5$ is divided by $u$ 
because we have $c_4+5=(c_1-c_3-2c_4-5)u$. 
We rewrite $u$ and $c_4+5$ as $a$ and $ab$, respectively, 
where $a$ is a non-zero integer and $b$ is an integer. 
Then we obtain 
$(c_1,c_2,c_3,c_4)=((-b+1)a+5, (3b-2)a+b-11, (-3b+1)a-b+10, ab-5)$. 

This completes the proof of the theorem. 
\end{proof}

\begin{rem}
The equality (B) in the proof of Theorem \ref{CS5} can be derived also from an argument similar to 
the second proof of Theorem \ref{CS4}. 
\end{rem}

\begin{rem}
Each family of polynomials in Case I has the `reciprocal' family in Case II: 
one family consists of the signed reciprocal polynomials of the polynomials in the other family. 
Such pairs of families are I-i-1 \& II-i-3, I-i-2 \& II-i-4, I-i-3 \& II-i-1, I-i-4 \& II-i-2, I-ii-1 \& II-ii-1, and 
I-ii-2 \& II-ii-2 (see Proposition \ref{CSdual}). 
\end{rem}

\begin{rem}
It is not difficult to check that 
a polynomial $f(x)$ in Cases I-i and II-i of Theorem \ref{CS5} satisfies the positivity condition 
if and only if $a$ satisfies the conditions shown below. 

\medskip

\begin{center}
\noindent
\begin{tabular}{|c|c||c|c||c|c||c|c|} \hline 
Case & positivity & Case & positivity & Case & positivity & Case & positivity \\ \hline 
I-i-1 & $a\leq 0$ & 
I-i-2 & $a\leq 0$ & 
I-i-3 & $a\leq -3$ & 
I-i-4 & $a\leq -2$ \\ \hline
II-i-1 & $a\leq -2$ & 
II-i-2 & $a\leq -1$ & 
II-i-3 & $a\leq -1$ & 
II-i-4 & $a\leq -1$ \\ \hline
\end{tabular} 
\end{center}

\medskip

For a polynomial $f(x)$ in Case I-ii-1 of Theorem \ref{CS5}, 
the condition (a) below implies the positivity, 
and the positivity implies the condition (b) below. 
\[
\mathrm{(a)} \quad  
b \geq 
\begin{cases}
\frac{1}{4}(a-1)^2+2 & (a\geq 3) \\
a & (a<3) 
\end{cases}
\quad 
\mathrm{(b)} \quad  
b > 
\begin{cases}
\frac{1}{4}(a-1)^2+\frac{3}{2} & (a\geq 3) \\
a-\frac{1}{2} & (a<3) 
\end{cases}
\]

For a polynomial $f(x)$ in Case I-ii-2 of Theorem \ref{CS5}, 
it satisfies the positivity condition if and only if the pair $(a,b)$ satisfies one of the following conditions: 

$\bullet$ $a\geq 1$ and $b\leq 0$; 

$\bullet$ $a\leq -1$ and $b\geq 1$; 

$\bullet$ $(a,b)=(1,1), (1,2), (1,3), (1,4), (1,5), (2,1), (2,2), (3,1), (4,1)$, \\ 
\ \ \ \ \ \ \ \ \ \ \ \ \ \ \ \ \ $(5,1), (6,1), (-1,0), (-1,-1), (-1,-2), (-2,0), (-2,-1), $ \\
\ \ \ \ \ \ \ \ \ \ \ \ \ \ \ \ \ $(-3,0), (-4,0), (-5,0), (-6,0).$

For polynomials $f(x)$ in Cases II-ii-1 and II-ii-2, 
we can find similar good conditions which are necessary/sufficient for the positivity condition. 
\end{rem}

\begin{rem}
The method of the third proof of Theorem \ref{CS4} is also useful for proving 
that a given monic polynomial 
of degree $5$ is a Cappell-Shaneson polynomial. 
For a polynomial $f(x)$ in Case I-ii-1 of Theorem \ref{CS5}, 
we have $f_p(x)-f_p^*(x)=x^2(x+1)$ for every prime number $p$. 
Since $f_p(0)=-1\ne 0$, there exists no polynomial of degree $2$ with coefficients 
in $\overline{\mathbb{F}}_p$ which divides both of $f_p(x)$ and $f_p^*(x)$. 
It implies that $f_p(x)$ is $2$-regular over $\mathbb{F}_p$ by Theorem \ref{2-regular}. 
(Note that $f(x)$ is $1$-regular over $\mathbb{Q}$ and $f_p(x)$ is $1$-regular over $\mathbb{F}_p$ 
because $f(1)=1$.) 
The companion matrix $A$ of $f(x)$ is a Cappell-Shaneson matrix 
by Proposition \ref{CS-regular_p} and hence 
$f(x)$ is a Cappell-Shaneson polynomial. 
\end{rem}

\begin{rem}
Cappell and Shaneson \cite{CS1976} found infinitely many polynomials in Cases I-ii-1 and II-ii-1 
of Theorem \ref{CS5}. All polynomials in Cases I-ii-1 and II-ii-1 were found by Gu and Jiang \cite{GJ1999}. 
\end{rem}


\section{Degree 6}


In this section we examine Cappell-Shaneson polynomials of degree $6$. 
In particular, we give a complete list of Cappell-Shaneson polynomials of degree $6$ 
for which the difference of the coefficients of $x^5$ and $x$ is less than or equal to $12$. 

\subsection{Systems of Diophantine equations}

Let $f(x)=x^6+c_5x^5+c_4x^4+c_3x^3+c_2x^2+c_1x+c_0$ be 
a monic polynomial of degree $6$ with integer coefficients. 
Let $A$ be the companion matrix of $f(x)$. 
Since $f(x)$ is the characteristic polynomial of $A$, $f(x)$ is a Cappell-Shaneson polynomial 
if and only if $A$ is a Cappell-Shaneson matrix by Corollary \ref{CSCS}. 
The condition $\det A=1$ is equivalent to the condition $c_0=1$. 
Since we have $\det(I-A)=f(1)=c_1+c_2+c_3+c_4+c_5+2$, 
$A$ satisfies the condition $\mathrm{CS}_1$ if and only if the equality 
\begin{equation}
c_1+c_2+c_3+c_4+c_5+2=\pm 1 \tag{A}
\end{equation}
holds. 
By direct computation of $\det(I-\bigwedge^2A)$, 
we know that $A$ satisfies the condition $\mathrm{CS}_2$ if and only if the equality 
\begin{equation}
\begin{split}
& c_1^4-c_1^3c_3+c_1^2c_2c_4-c_1^2c_4^2-2c_1^3c_5-c_1c_2^2c_5+3c_1^2c_3c_5 
+c_1c_4^2c_5+c_2^2c_5^2 \\
& -3c_1c_3c_5^2-c_2c_4c_5^2+2c_1c_5^3+c_3c_5^3-c_5^4-3c_1^2c_2+c_2^3+3c_1^2c_4-3c_2^2c_4 \\
& +3c_2c_4^2-c_4^3+6c_1c_2c_5-6c_1c_4c_5-3c_2c_5^2+3c_4c_5^2=\pm 1 
\end{split}\tag{B}
\end{equation}
holds. 
By direct computation of $\det(I-\bigwedge^3A)$, 
we know that $A$ satisfies the condition $\mathrm{CS}_3$ if and only if the equality 
\begin{equation}
\begin{split}
& c_1^3+c_1^2c_4+c_1c_3c_5+c_2c_5^2+c_5^3-4c_1c_2+c_3^2 \\
& -4c_2c_4-2c_1c_5-4c_4c_5+4c_3+4=\pm 1
\end{split}\tag{C}
\end{equation}
holds. 
Consequently, $f(x)$ is a Cappell-Shaneson polynomial if and only if the $6$-tuple 
$(c_0,c_1,c_2,c_3,c_4,c_5)$ of its coefficients satisfies the system of Diophantine equations 
$c_0=1$, (A), (B) and (C). 

Let $\varepsilon_1,\varepsilon_2,\varepsilon_3$ be the right-hand side of the equalities (A), (B) and (C), 
respectively. Hence each of $\varepsilon_1,\varepsilon_2,\varepsilon_3$ is equal to $+1$ or $-1$. 
From the equality (A), we have 
\begin{equation}
c_3=-c_1-c_2-c_4-c_5-2+\varepsilon_1.  \tag{A'}
\end{equation}
We put $p:=c_4-c_2$ and $q:=c_5-c_1$. 
Using these equalities, we can rewrite the left-hand side of the equality (B), and obtain the following 
equality which is equivalent to (B). 
\[
(p+2q)(q(p-2q)c_1-q^2c_2-p^2+2pq-q^3-q^2)+\varepsilon_1q^3=\varepsilon_2
\]
If we put $w:=q(p-2q)c_1-q^2c_2-p^2+2pq-q^3-q^2$, this is equivalent to the following equality. 
\begin{equation}
(p+2q)w=\varepsilon_2-\varepsilon_1q^3 \tag{B'}
\end{equation}
Rewriting the left-hand side of the equality (C) in a similar way, we obtain the next equality. 
\begin{equation}
\varepsilon_1(c_1^2+(q-4)c_1-4c_2-2p-2q)-w+1=\varepsilon_3 \tag{C'}
\end{equation}
Consequently, $f(x)$ is a Cappell-Shaneson polynomial if and only if the $6$-tuple 
$(c_0,c_1,c_2,c_3,c_4,c_5)$ of its coefficients satisfies the system of Diophantine equations $c_0=1$, 
(A'), (B') and (C'). 

\subsection{Cappell-Shaneson polynomials for small $\boldsymbol{|c_5-c_1|}$}

We solve the system of Diophantine equations $c_0=1$, (A'), (B') and (C') 
when $|c_5-c_1|$ is small. 

\begin{prop}\label{CS6}
Let $f(x)=x^6+c_5x^5+c_4x^4+c_3x^3+c_2x^2+c_1x+c_0$ be 
a monic polynomial of degree $6$ with integer coefficients. 
If $c_1$ and $c_5$ satisfy the inequality $0\leq c_5-c_1\leq 12$, 
then $f(x)$ is a Cappell-Shaneson polynomial if and only if the $6$-tuple $(c_0,c_1,c_2,c_3,c_4,c_5)$ of 
its coefficients is equal to one of those exhibited in the table of Appendix \ref{app}. 
Further, $f(x)$ is positive if and only if its coefficients satisfy 
the positivity condition indicated in each row in the table of Appendix \ref{app}. 
\end{prop}

For a polynomial $f(x)$ which satisfies the inequality $-12\leq c_5-c_1\leq -1$, 
the signed reciprocal polynomial $f^*(x)$ satisfies the inequality $1\leq c_5-c_1\leq 12$, 
and $f(x)$ is a Cappell-Shaneson polynomial if and only if $f^*(x)$ is a Cappell-Shaneson polynomial. 
We thus obtain a complete list of Cappell-Shaneson polynomials of degree $6$ 
which satisfy the inequality $-12\leq c_5-c_1\leq 12$ from Proposition \ref{CS6}. 

We will not give the full proof of Proposition \ref{CS6} because it consists of many individual 
considerations of solutions of the above system of Diophantine equations. 
Instead, we show how to solve the system of equations in several special cases, 
which would be enough for the reader to recover the whole proof. 

\begin{exmp}
Suppose that $c_5-c_1=0$ and $\varepsilon_1=\varepsilon_2=1$. 
Since we have $q=0$, the equality (B') is equivalent to $-p^3=\varepsilon_2=1$. 
Hence we have $p=-1$, 
and then $c_4=c_2-1$, $c_5=c_1$, and $c_3=-2c_1-2c_2-1+\varepsilon_1=-2c_1-2c_2$. 
From the equation (C'), we obtain $(c_1-2)^2-4c_2=\varepsilon_3$, 
which has integral solutions if and only if $(c_1-2)^2\equiv\varepsilon_3\; (\mathrm{mod}\; 4)$. 
This congruence has solutions if and only if $\varepsilon_3=1$, and 
every solution for $\varepsilon_3=1$ is expressed as $c_1=2a+1$, where $a$ is an integer. 
Hence we obtain 
\[
(c_0,c_1,c_2,c_3,c_4,c_5)=(1,\, 2a+1,\, a^2-a,\, -2a^2-2a-2,\, a^2-a-1,\, 2a+1)
\]
for some integer $a$. 
\end{exmp}

\begin{exmp}
Suppose that $c_5-c_1=1$ and $\varepsilon_1=\varepsilon_2=1$. 
Since we have $q=1$, the equalities (B') and (C') are equivalent to $(p+2)w=0$ 
and $c_1^2-3c_1-4c_2-2p-1-w=\varepsilon_3$, respectively, 
where $w=(p-2)c_1-c_2-p^2+2p-2$. 
Note that $c_5=c_1+1$ by assumption, and $c_4=c_2+p$ by definition. 
We then have $c_3=-2c_1-2c_2-p-2$. 

Suppose that $p=-2$. We have $c_4=c_2-2$ and $c_3=-2c_1-2c_2$. 
The equation (C') is equivalent to $c_1^2+c_1-3c_2+13=\varepsilon_3$, 
which has integral solutions if and only if $c_1^2+c_1+13\equiv\varepsilon_3\; (\mathrm{mod}\; 3)$. 
This congruence has solutions if and only if $\varepsilon_3=1$, and 
every solution for $\varepsilon_3=1$ is expressed as $c_1=3a-1$ or $c_1=3a$, where $a$ is an integer. 
Hence we obtain 
\begin{align*}
(c_0,c_1,c_2,c_3,c_4,c_5) & =(1,\, 3a-1,\, 3a^2-a+4,\, -6a^2-4a-6,\, 3a^2-a+2,\, 3a), \\
& =(1,\, 3a,\, 3a^2+a+4,\, -6a^2-8a-8,\, 3a^2+a+2,\, 3a+1)
\end{align*}
for some integer $a$. 

Suppose that $w=0$, which implies $c_2=(p-2)c_1-p^2+2p-2$, and 
$c_1^2-3c_1-4c_2-2p-1=\varepsilon_3$. 
Eliminating $c_2$ from these equalities, we obtain 
$(c_1-2p+2)(c_1-2p+3)=\varepsilon_3-1$. 
This equation has integral solutions if and only if $\varepsilon_3=1$, and 
every solution for $\varepsilon_3=1$ is expressed as $c_1=2p-3$ or $c_1=2p-2$. 
Computing $c_2,c_3,c_4,c_5$ from $c_1$ and substituting $a+1$ for $p$, we obtain
\begin{align*}
(c_0,c_1,c_2,c_3,c_4,c_5) & =(1,\, 2a-1,\, a^2-3a,\, -2a^2+a-1,\, a^2-2a+1,\, 2a), \\
& =(1,\, 2a,\, a^2-2a-1,\, -2a^2-a-1,\, a^2-a,\, 2a+1). 
\end{align*}
\end{exmp}

\begin{exmp}
Suppose that $c_5-c_1=3$ and $\varepsilon_1=\varepsilon_2=1$. 
Since we have $q=3$, we obtain $(p+6)w=-26$ from the equality (B'), 
and $c_1^2-c_1-4c_2-2p-5-w=\varepsilon_3$ from the equation (C'). 
Since $p+6$ is a divisor of $-26$, it must be $\pm 1,\pm 2,\pm 13$, or $\pm 26$. 
Considering the equality (B') modulo $3$, we have the congruence 
$(p+6)^3\equiv p^3\equiv -1\, (\textrm{mod}\; 3)$. Therefore $p+6$ must be $2,26,-1$, or $-13$. 

Suppose that $p+6=2$. 
We have $w=-13$ and $p=-4$, and then $w=-30c_1-9c_2-76$ and $10c_1+3c_2+21=0$. 
This equality together with the equation (C') implies $3c_1^2+37c_1+132=3\varepsilon_3$, 
which has no integral solutions. 
Therefore we have no solution $(c_0,c_1,c_2,c_3,c_4,c_5)$ in this case. 

Suppose that $p+6=26$. 
We have $w=-1$ and $p=20$, and then $w=42c_1-9c_2-316$ and $14c_1-3c_2-105=0$. 
This equality together with the equation (C') implies $3c_1^2-59c_1+288=3\varepsilon_3$, 
which has no integral solutions. 
Therefore we have no solution $(c_0,c_1,c_2,c_3,c_4,c_5)$ in this case. 

Suppose that $p+6=-1$. 
We have $w=26$ and $p=-7$, and then $w=-39c_1-9c_2-127$ and $13c_1+3c_2+51=0$. 
This equality together with the equation (C') implies $(c_1+12)(3c_1+13)=3(1+\varepsilon_3)$, 
which has integral solutions if and only if $\varepsilon_3=-1$, and 
the solution for $\varepsilon_3=-1$ is $c_1=-12$. 
Thus we obtain 
\[
(c_0,c_1,c_2,c_3,c_4,c_5)=(1,-12,35,-43,28,-9). 
\]

Suppose that $p+6=-13$. 
We have $w=2$ and $p=-19$, and then $w=-75c_1-9c_2-511$ and $25c_1+3c_2+171=0$. 
This equality together with the equation (C') implies $3c_1^2+97c_1+777=3\varepsilon_3$, 
whose solution is $c_1=-15$ if $\varepsilon_3=-1$ and $c_1=-18$ if $\varepsilon_3=1$. 
Thus we obtain 
\[
(c_0,c_1,c_2,c_3,c_4,c_5)=(1,-15,68,-91,49,-12), (1,-18,93,-135,74,-15). 
\]
\end{exmp}

\begin{rem}
There exist infinitely many Cappell-Shaneson polynomials of degree $6$ with $c_5-c_1=q$ 
if $q$ is equal to $1$, $0$, or $-1$, while there exist only finitely many such polynomials 
if $|q|$ is greater than one. 
Gu and Jiang \cite{GJ1999} found 
all polynomials in the first row in the table of Appendix \ref{app}. 
\end{rem}

Although it is possible to carry out a similar computation for each $q$ greater than $12$, 
the more the number of divisors of $q^3\pm 1$ increases, 
the more complicated the computation for such a $q$ becomes. 

\subsection{Other families of Cappell-Shaneson polynomials}

We prove that there exist at least four Cappell-Shaneson polynomials of degree $6$ 
with $c_5-c_1=q$ for every integer $q$. 

\begin{prop}\label{hol}
Let $f(x)=x^6+c_5x^5+c_4x^4+c_3x^3+c_2x^2+c_1x+c_0$ be 
a monic polynomial of degree $6$ with integer coefficients. 
For every integer $q$, if the $6$-tuple $(c_0,c_1,c_2,c_3,c_4,c_5)$ of 
its coefficients is equal to one of those exhibited below, 
then $f(x)$ is a Cappell-Shaneson polynomial which satisfies $c_5-c_1=q$. 
\begin{small}
\begin{longtable}{|c||c|c|c|c|c|c|} \hline 
$c_5\! -\! c_1$ & $c_0$ & $c_1$ & $c_2$ & $c_3$ & $c_4$ & $c_5$ \\ \hline 
$q$ & $1$ & $-3q-9$ & $2q^2+15q+30$ & $-3q^2-22q-42$ & $q^2+12q+29$ & $-2q-9$ \\ \cline{2-7}
& $1$ & $-2q-3$ & $q^2+4q+5$ & $-q^2-4q-6$ & $3q+4$ & $-q-3$ \\ \cline{2-7}
& $1$ & $2q-9$ & $q^2-12q+29$ & $-3q^2+22q-42$ & $2q^2-15q+30$ & $3q-9$ \\ \cline{2-7}
& $1$ & $q-3$ & $-3q+4$ & $-q^2+4q-6$ & $q^2-4q+5$ & $2q-3$ \\ \hline
\end{longtable}
\end{small}
\end{prop}

\begin{proof}
It is not difficult to see that the coefficients of $f(x)$ 
satisfy the equalities $c_0=1$, (A'), (B') and (C') 
for every $6$-tuple $(c_0,c_1,c_2,c_3,c_4,c_5)$ in the table above. 
\end{proof}

\begin{rem}
It is not difficult to check that 
the polynomial $f(x)$ in the first, second, third and fourth row 
in the table of Proposition \ref{hol} satisfies the positivity condition 
if and only if $q$ satisfies the condition $q\geq -4$, $q\geq -2$, $q\leq 4$ and $q\leq 2$, respectively. 
\end{rem}

As shown in the table of Appendix \ref{app}, the number of Cappell-Shaneson polynomials of degree $6$ 
with $c_5-c_1=q$ is equal to $4$ if $q$ is equal to $4$, $5$, $8$, $10$, $11$, or $12$. 
We now pose the following problem. 

\begin{prob}\label{exist}
Do there exist infinitely many $q$ for which the number of Cappell-Shaneson polynomials of degree $6$ 
with $c_5-c_1=q$ is equal to $4$?
\end{prob}

Further computation tells us that 
the number of Cappell-Shaneson polynomials of degree $6$ 
with $c_5-c_1=q$ is equal to $4$ if $q$ is equal to 
$15$, $16$, $17$, $20$, $22$, $23$, $24$, $29$, $30$, $32$, $33$, $34$, or $40$. 

We give a definition of basic Cappell-Shaneson polynomials. 

\begin{defn}
Let $f(x)=x^6+c_5x^5+c_4x^4+c_3x^3+c_2x^2+c_1x+c_0$ be a Cappell-Shaneson polynomial of 
degree $6$. We put $p:=c_4-c_2$, $q:=c_5-c_1$, $w:=q(p-2q)c_1-q^2c_2-p^2+2pq-q^3-q^2$, 
and consider the elements $\varepsilon_1$, $\varepsilon_2$ of $\{1,-1\}$ 
defined by the equalities (A') and (B'). We assume that $q\geq 2$. 
A divisor $d$ of $\varepsilon_2-\varepsilon_1q^3$ is called {\it basic} if 

$\bullet$ $d$ is equal to 
$\pm 1$, $\pm (q-1)$, $\pm (q^2+q+1)$, or $\pm (q^3-1)$ if $\varepsilon_1=\varepsilon_2$, 

$\bullet$ 
$d$ is equal to $\pm 1$, $\pm (q+1)$, $\pm (q^2-q+1)$, or $\pm (q^3+1)$ if 
$\varepsilon_1\ne\varepsilon_2$. 

\noindent
We call $f(x)$ is {\it basic} if $p+2q$ is a basic divisor of 
$\varepsilon_2-\varepsilon_1q^3$. 
\end{defn}

All Cappell-Shaneson polynomials in the table of Proposition \ref{hol} are basic. 
In order to solve Problem \ref{exist} affirmatively, it is enough to show that there exist neither 
basic Cappell-Shaneson polynomials with $c_5-c_1=q$ other than those 
in the table of Proposition \ref{hol} 
nor non-basic Cappell-Shaneson polynomials with $c_5-c_1=q$ for infinitely many $q$. 

\begin{rem}
It follows from Faltings' theorem \cite{Faltings1983} 
that there exist only finitely many basic Cappell-Shaneson polynomials 
other than those in the table of Proposition \ref{hol}. 
Moreover, if the solutions of four Diophantine equations 

(i) $x^2y^2-x^3+x^2y+4x+4y+4=1$, 

(ii) $x^2y^2-x^3+x^2y+4x+4y+4=-1$, 

(iii) $x^2y^2-x^3+x^2y-4x+4y=1$, 

(iv) $x^2y^2-x^3+x^2y-4x+4y=-1$ 

\noindent
are equal to 

(i) $(x,y)=(-1,-5), (-1,0), (1,-3), (1,-2)$, 

(ii) $(x,y)=(3,-2)$, 

(iii) $(x,y)=(-1,-4), (-1,-1), (1,-6), (1,1)$, 

(iv) $(x,y)=(-1,-3), (-1,-2)$, 

\noindent
respectively, then there exist only four basic Cappell-Shaneson polynomials with $c_5-c_1\geq 3$ 
other than those in the table of Proposition \ref{hol}. 
Several methods 
for determining the set of rational points on a given algebraic curve might be useful 
to show that the solutions (i)--(iv) are all integral solutions of the equations (i)--(iv) 
(see \cite{BPS2016} and \cite{BMSST2008}). 
\end{rem}


\section{Higher degrees}


In this section we discuss Cappell-Shaneson polynomials of degree greater than or equal to $7$. 

\subsection{Degree $\boldsymbol{7}$}

Let $f(x)=x^7+c_6x^6+c_5x^5+c_4x^4+c_3x^3+c_2x^2+c_1x+c_0$ be 
a monic polynomial of degree $7$ with integer coefficients. 
Let $A$ be the companion matrix of $f(x)$. 
Since $f(x)$ is the characteristic polynomial of $A$, it is a Cappell-Shaneson polynomial 
if and only if $A$ is a Cappell-Shaneson matrix by Corollary \ref{CSCS}. 
The condition $\det A=1$ is equivalent to the condition $c_0=-1$. 
Since we have $\det(I-A)=f(1)=c_1+c_2+c_3+c_4+c_5+c_6$, 
the matrix $A$ satisfies the condition $\mathrm{CS}_1$ if and only if the equality 
\begin{equation*}
c_1+c_2+c_3+c_4+c_5+c_6=\pm 1 
\end{equation*}
holds. 
If we wrote down $\det(I-\bigwedge^2A)$ and $\det(I-\bigwedge^3A)$ as polynomials in $c_1$, 
$c_2$, $c_3$, $c_4$, $c_5$ and $c_6$, they would span several pages of this paper. 
Here we assume that $f(x)$ satisfies the equalities $c_1+c_6=0$ and $c_2+c_5=0$. 
Combining these equalities with the condition $\mathrm{CS}_1$, 
the determinants $\det(I-\bigwedge^2A)$ and $\det(I-\bigwedge^3A)$ are expressed as polynomials in 
$c_1$, $c_2$ and $c_3$ (see Appendix \ref{cs23}). 

\begin{prop}\label{CS7}
Let $f(x)=x^7+c_6x^6+c_5x^5+c_4x^4+c_3x^3+c_2x^2+c_1x+c_0$ be 
a monic polynomial of degree $7$ with integer coefficients. 
If the $7$-tuple $(c_0,c_1,c_2,c_3,c_4$, $c_5,c_6)$ of 
its coefficients is equal to one of those exhibited below, 
where $a$ stands for an arbitrary integer, 
then $f(x)$ is a Cappell-Shaneson polynomial which satisfies $c_1+c_6=c_2+c_5=0$. 
Further, $f(x)$ is positive if and only if its coefficients satisfy 
the positivity condition indicated in each row in the table below. 
\begin{small}
\begin{longtable}{|c|c|c|c|c|c|c|c|} \hline 
$c_0$ & $c_1$ & $c_2$ & $c_3$ & $c_4$ & $c_5$ & $c_6$ & {\rm positivity} \\ \hline 
$-1$ & $-1$ & $a$ & $a+1$ & $-a$ & $-a$ & $1$ &  $a\leq 5$ \\ \hline
$-1$ & $-1$ & $a$ & $a$ & $-a+1$ & $-a$ & $1$ &  $a\notin \mathbb{Z}$ \\ \hline
$-1$ & $-1$ & $a$ & $a-1$ & $-a$ & $-a$ & $1$ &  $a\notin \mathbb{Z}$ \\ \hline
$-1$ & $-1$ & $a$ & $a$ & $-a-1$ & $-a$ & $1$ &  $a\leq 5$ \\ \hline
$-1$ & $a$ & $-a+2$ & $a^2+3$ & $-a^2-2$ & $a-2$ & $-a$ &  $a\geq -2$ \\ \hline
$-1$ & $a$ & $-a+2$ & $a^2+2$ & $-a^2-1$ & $a-2$ & $-a$ &  $a\geq 0$ \\ \hline
$-1$ & $a$ & $-a+2$ & $a^2+1$ & $-a^2-2$ & $a-2$ & $-a$ &  $a\geq 0$ \\ \hline
$-1$ & $a$ & $-a+2$ & $a^2+2$ & $-a^2-3$ & $a-2$ & $-a$ &  $a\geq -2$ \\ \hline
\end{longtable} 
\end{small}
\end{prop}

\begin{proof}
It is not difficult to see that the coefficients of $f(x)$ 
satisfy the equalities $c_0=-1$, $c_1+c_2+c_3+c_4+c_5+c_6=\pm 1$, 
$\det(I-\bigwedge^2A)=\pm 1$ and $\det(I-\bigwedge^3A)=\pm 1$ 
for every $7$-tuple $(c_0,c_1,c_2,c_3,c_4,c_5,c_6)$ in the table above. 
See Appendix \ref{cs23}. 
\end{proof}

\begin{rem}
Gu and Jiang \cite{GJ1999} found 
all polynomials in the first row in the table of Proposition \ref{CS7}. 
\end{rem}

It is not clear to the authors whether 
there exist many Cappell-Shaneson polynomials $f(x)$ of degree $7$ 
which do not satisfy the condition $c_1+c_6=c_2+c_5=0$. 

\subsection{Degree $\boldsymbol{8}$ and higher degrees}

Let $f(x)=x^8+c_7x^7+c_6x^6+c_5x^5+c_4x^4+c_3x^3+c_2x^2+c_1x+1$ be 
a doubly monic polynomial of degree $8$ with integer coefficients. 
By following the next three steps, 
we can verify that a given $f(x)$ is not a Cappell-Shaneson polynomial. 

Step 1: Choose a prime number $p$, and factor $f_p(x)$ over $\mathbb{F}_p$. 

Step 2: Since the algebraic closure of $\mathbb{F}_p$ is equal to $\bigcup_{i=1}^{\infty}\mathbb{F}_{p^i}$, 
there exists a positive integer $m$ such that $f_p(x)$ is decomposable in $\mathbb{F}_{p^m}$. 
Find such an integer $m$ and the roots of $f_p(x)$ in $\mathbb{F}_{p^m}$. 

Step 3: Compute all possible products of roots of $f_p(x)$ of length less than $5$, 
and check whether any of them is equal to one, in which case $f(x)$ is not a Cappell-Shaneson polynomial 
by Proposition \ref{CS-regular_p}. 

Using the software system {\sf SageMath}, it was confirmed that there exists 
no Cappell-Shaneson polynomial $f(x)$ of degree $8$ with $-6\leq c_1,\ldots ,c_7\leq 6$. 
The last polynomial which was checked is the polynomial $f(x)=x^8-2x^7-3x^6+3x^5-5x^4+6x^3-4x^2+4x+1$. 
It was detected not to be a Cappell-Shaneson polynomial with respect to the prime number 
$p=5525329$. This method is also useful for polynomials of degree $9$ or higher. 



\appendix

\begin{small}


\section{A list of Cappell-Shaneson polynomials of degree $6$}\label{app}


The following is a complete list of Cappell-Shaneson polynomials 
$f(x)=x^6+c_5x^5+c_4x^4+c_3x^3+c_2x^2+c_1x+c_0$ of degree $6$ which satisfy 
the inequality $0\leq q=c_5-c_1\leq 12$. The symbol $a$ stands for an arbitrary integer. 

\begin{longtable}{|c||c|c|c|c|c|c|c|} \hline 
$q$ & $c_0$ & $c_1$ & $c_2$ & $c_3$ & $c_4$ & $c_5$ & positivity \\ \hline 
$0$ & $1$ & $2a+1$ & $a^2-a$ & $-2a^2-2a-2$ & $a^2-a-1$ & $2a+1$ & $a\leq 0$ \\ \cline{2-8}
& $1$ & $2a+1$ & $a^2-a-1$ & $-2a^2-2a-2$ & $a^2-a-2$ & $2a+1$ & $a\leq -1$ \\ \cline{2-8}
& $1$ & $2a+1$ & $a^2-a-1$ & $-2a^2-2a-2$ & $a^2-a$ & $2a+1$ & $a\leq 0$ \\ \cline{2-8}
& $1$ & $2a+1$ & $a^2-a-2$ & $-2a^2-2a-2$ & $a^2-a-1$ & $2a+1$ & $a\leq -1$ \\ \hline
$1$ & $1$ & $3a-1$ & $3a^2-a+4$ & $-6a^2-4a-6$ & $3a^2-a+2$ & $3a$ & $a\in\mathbb{Z}$ \\ \cline{2-8}
& $1$ & $3a$ & $3a^2+a+4$ & $-6a^2-8a-8$ & $3a^2+a+2$ & $3a+1$ & $a\in\mathbb{Z}$ \\ \cline{2-8}
& $1$ & $2a-1$ & $a^2-3a$ & $-2a^2+a-1$ & $a^2-2a+1$ & $2a$ & $a\in\mathbb{Z}$ \\ \cline{2-8}
& $1$ & $2a$ & $a^2-2a-1$ & $-2a^2-a-1$ & $a^2-a$ & $2a+1$ & $a\leq 0$ \\ \cline{2-8}
& $1$ & $-5$ & $10$ & $-11$ & $7$ & $-4$ & Yes \\ \cline{2-8}
& $1$ & $-12$ & $45$ & $-67$ & $42$ & $-11$ & Yes \\ \cline{2-8}
& $1$ & $-2$ & $3$ & $-3$ & $2$ & $-1$ & Yes \\ \cline{2-8}
& $1$ & $-7$ & $18$ & $-23$ & $17$ & $-6$ & Yes \\ \cline{2-8}
& $1$ & $5a-2$ & $5a^2-11a+2$ & $-10a^2+12a-2$ & $5a^2-11a$ & $5a-1$ & $a\leq 0$ \\ \cline{2-8}
& $1$ & $5a-1$ & $5a^2-9a$ & $-10a^2+8a$ & $5a^2-9a-2$ & $5a$ & $a\leq 0$ \\ \cline{2-8}
& $1$ & $5a$ & $5a^2-7a-2$ & $-10a^2+4a+2$ & $5a^2-7a-4$ & $5a+1$ & $a\leq -1$ \\ \cline{2-8}
& $1$ & $5a+2$ & $5a^2-3a-4$ & $-10a^2-4a+2$ & $5a^2-3a-6$ & $5a+3$ & $a\leq -1$ \\ \cline{2-8}
& $1$ & $2a-1$ & $a^2-3a$ & $-2a^2+a-2$ & $a^2-2a$ & $2a$ & $a\leq 0$ \\ \cline{2-8}
& $1$ & $2a$ & $a^2-2a-1$ & $-2a^2-a-3$ & $a^2-a$ & $2a+1$ & $a\leq 2$ \\ \cline{2-8}
& $1$ & $2a-1$ & $a^2-3a$ & $-2a^2+a-3$ & $a^2-2a+1$ & $2a$ & $a\leq 3$ \\ \cline{2-8}
& $1$ & $2a$ & $a^2-2a-2$ & $-2a^2-a-2$ & $a^2-a$ & $2a+1$ & $a\leq 0$ \\ \hline
$2$ & $1$ & $-5$ & $11$ & $-12$ & $8$ & $-3$ & Yes \\ \cline{2-8}
& $1$ & $-7$ & $18$ & $-22$ & $15$ & $-5$ & Yes \\ \cline{2-8}
& $1$ & $-13$ & $53$ & $-72$ & $42$ & $-11$ & Yes \\ \cline{2-8}
& $1$ & $-15$ & $68$ & $-98$ & $57$ & $-13$ & Yes \\ \cline{2-8}
& $1$ & $1$ & $-4$ & $-4$ & $1$ & $3$ & No \\ \cline{2-8}
& $1$ & $3$ & $-3$ & $-10$ & $2$ & $5$ & No \\ \cline{2-8}
& $1$ & $-7$ & $17$ & $-18$ & $10$ & $-5$ & Yes \\ \cline{2-8}
& $1$ & $-13$ & $50$ & $-72$ & $43$ & $-11$ & Yes \\ \cline{2-8}
& $1$ & $-3$ & $4$ & $-4$ & $3$ & $-1$ & Yes \\ \cline{2-8}
& $1$ & $-5$ & $9$ & $-10$ & $8$ & $-3$ & Yes \\ \cline{2-8}
& $1$ & $-7$ & $15$ & $-14$ & $10$ & $-5$ & Yes \\ \cline{2-8}
& $1$ & $-9$ & $24$ & $-28$ & $19$ & $-7$ & Yes \\ \cline{2-8}
& $1$ & $-15$ & $69$ & $-98$ & $56$ & $-13$ & Yes \\ \cline{2-8}
& $1$ & $-17$ & $86$ & $-128$ & $73$ & $-15$ & Yes \\ \cline{2-8}
& $1$ & $1$ & $-3$ & $-4$ & $0$ & $3$ & No \\ \cline{2-8}
& $1$ & $-1$ & $-2$ & $-2$ & $1$ & $1$ & Yes \\ \hline
$3$ & $1$ & $-12$ & $35$ & $-43$ & $28$ & $-9$ & Yes \\ \cline{2-8}
& $1$ & $-15$ & $68$ & $-91$ & $49$ & $-12$ & Yes \\ \cline{2-8}
& $1$ & $-18$ & $93$ & $-135$ & $74$ & $-15$ & Yes \\ \cline{2-8}
& $1$ & $3$ & $-4$ & $-12$ & $4$ & $6$ & No \\ \cline{2-8}
& $1$ & $-9$ & $26$ & $-27$ & $13$ & $-6$ & Yes \\ \cline{2-8}
& $1$ & $-15$ & $64$ & $-91$ & $51$ & $-12$ & Yes \\ \cline{2-8}
& $1$ & $-3$ & $2$ & $-3$ & $3$ & $0$ & Yes \\ \cline{2-8}
& $1$ & $-6$ & $10$ & $-4$ & $2$ & $-3$ & Yes \\ \cline{2-8}
& $1$ & $-12$ & $38$ & $-48$ & $30$ & $-9$ & Yes \\ \cline{2-8}
& $1$ & $-15$ & $68$ & $-90$ & $48$ & $-12$ & Yes \\ \cline{2-8}
& $1$ & $-18$ & $94$ & $-136$ & $74$ & $-16$ & Yes \\ \cline{2-8}
& $1$ & $0$ & $-5$ & $-3$ & $2$ & $3$ & No \\ \hline
$4$ & $1$ & $-21$ & $122$ & $-178$ & $93$ & $-17$ & Yes \\ \cline{2-8}
& $1$ & $-11$ & $37$ & $-38$ & $16$ & $-7$ & Yes \\ \cline{2-8}
& $1$ & $-1$ & $-3$ & $-2$ & $2$ & $3$ & Yes \\ \cline{2-8}
& $1$ & $1$ & $-8$ & $-6$ & $5$ & $5$ & No \\ \hline
$5$ & $1$ & $-24$ & $155$ & $-227$ & $114$ & $-19$ & Yes \\ \cline{2-8}
& $1$ & $-13$ & $50$ & $-51$ & $19$ & $-8$ & Yes \\ \cline{2-8}
& $1$ & $1$ & $-6$ & $-7$ & $5$ & $6$ & No \\ \cline{2-8}
& $1$ & $2$ & $-11$ & $-11$ & $10$ & $7$ & No \\ \hline
$6$ & $1$ & $-27$ & $192$ & $-282$ & $137$ & $-21$ & Yes \\ \cline{2-8}
& $1$ & $-15$ & $65$ & $-66$ & $22$ & $-9$ & Yes \\ \cline{2-8}
& $1$ & $-3$ & $-5$ & $24$ & $-22$ & $3$ & No \\ \cline{2-8}
& $1$ & $3$ & $-7$ & $-18$ & $12$ & $9$ & No \\ \cline{2-8}
& $1$ & $3$ & $-14$ & $-18$ & $17$ & $9$ & No \\ \hline
$7$ & $1$ & $-30$ & $233$ & $-343$ & $162$ & $-23$ & Yes \\ \cline{2-8}
& $1$ & $-17$ & $82$ & $-83$ & $25$ & $-10$ & Yes \\ \cline{2-8}
& $1$ & $5$ & $-6$ & $-35$ & $23$ & $12$ & No \\ \cline{2-8}
& $1$ & $4$ & $-17$ & $-27$ & $26$ & $11$ & No \\ \cline{2-8}
& $1$ & $-21$ & $78$ & $-106$ & $60$ & $-14$ & Yes \\ \cline{2-8}
& $1$ & $-30$ & $248$ & $-346$ & $148$ & $-23$ & Yes \\ \cline{2-8}
& $1$ & $18$ & $56$ & $-228$ & $128$ & $25$ & No \\ \cline{2-8}
& $1$ & $-22$ & $81$ & $-111$ & $64$ & $-15$ & Yes \\ \hline
$8$ & $1$ & $-33$ & $278$ & $-410$ & $189$ & $-25$ & Yes \\ \cline{2-8}
& $1$ & $-19$ & $101$ & $-102$ & $28$ & $-11$ & Yes \\ \cline{2-8}
& $1$ & $7$ & $-3$ & $-58$ & $38$ & $15$ & No \\ \cline{2-8}
& $1$ & $5$ & $-20$ & $-38$ & $37$ & $13$ & No \\ \hline
$9$ & $1$ & $-36$ & $327$ & $-483$ & $218$ & $-27$ & Yes \\ \cline{2-8}
& $1$ & $-21$ & $122$ & $-123$ & $31$ & $-12$ & Yes \\ \cline{2-8}
& $1$ & $9$ & $2$ & $-87$ & $57$ & $18$ & No \\ \cline{2-8}
& $1$ & $6$ & $-23$ & $-51$ & $50$ & $15$ & No \\ \cline{2-8}
& $1$ & $34$ & $246$ & $-734$ & $410$ & $43$ & No \\ \cline{2-8}
& $1$ & $1$ & $-28$ & $67$ & $-51$ & $10$ & No \\ \cline{2-8}
& $1$ & $-31$ & $116$ & $-162$ & $96$ & $-22$ & Yes \\ \cline{2-8}
& $1$ & $-46$ & $566$ & $-852$ & $366$ & $-37$ & Yes \\ \hline
$10$ & $1$ & $-39$ & $380$ & $-562$ & $249$ & $-29$ & Yes \\ \cline{2-8}
& $1$ & $-23$ & $145$ & $-146$ & $34$ & $-13$ & Yes \\ \cline{2-8}
& $1$ & $11$ & $9$ & $-122$ & $80$ & $21$ & No \\ \cline{2-8}
& $1$ & $7$ & $-26$ & $-66$ & $65$ & $17$ & No \\ \hline
$11$ & $1$ & $-42$ & $437$ & $-647$ & $282$ & $-31$ & Yes \\ \cline{2-8}
& $1$ & $-25$ & $170$ & $-171$ & $37$ & $-14$ & Yes \\ \cline{2-8}
& $1$ & $13$ & $18$ & $-163$ & $107$ & $24$ & No \\ \cline{2-8}
& $1$ & $8$ & $-29$ & $-83$ & $82$ & $19$ & No \\ \hline
$12$ & $1$ & $-45$ & $498$ & $-738$ & $317$ & $-33$ & Yes \\ \cline{2-8}
& $1$ & $-27$ & $197$ & $-198$ & $40$ & $-15$ & Yes \\ \cline{2-8}
& $1$ & $15$ & $29$ & $-210$ & $138$ & $27$ & No \\ \cline{2-8}
& $1$ & $9$ & $-32$ & $-102$ & $101$ & $21$ & No \\ \hline
\end{longtable} 


\section{The conditions $\mathrm{CS}_2$ and $\mathrm{CS}_3$ in degree $7$}\label{cs23}


Let $f(x)=x^7+c_6x^6+c_5x^5+c_4x^4+c_3x^3+c_2x^2+c_1x+c_0$ be 
a monic polynomial of degree $7$ with $c_0=-1$ and $c_1+c_6=c_2+c_5=0$, and 
$A$ the companion matrix of $f(x)$. 
We assume that $f(x)$ satisfies the equality $c_1+c_2+c_3+c_4+c_5+c_6=\pm 1$. 
Then $\det(I-\bigwedge^2A)$ and 
$\det(I-\bigwedge^3A)$ are expressed as polynomials in $c_1$, $c_2$, $c_3$ as follows. 



If $c_1+c_2+c_3+c_4+c_5+c_6=1$, then $\det(I-\bigwedge^2A)=1$, and 
{\allowdisplaybreaks %
\begin{align*}
\det(I-\bigwedge^3A)
= & \; c_1^{7}+c_1^{6}+6c_1^{5}c_2-2c_1^{5}c_3-9c_1^{5}+4c_1^{4}c_2+9c_1^{3}c_2^{2}+8c_1^{4}c_3
     -10c_1^{3}c_2c_3 \\
     + & \, c_1^{3}c_3^{2}-30c_1^{4}-19c_1^{3}c_2-c_1^{2}c_2^{2}+4c_1c_2^{3}+35c_1
     ^{3}c_3+18c_1^{2}c_2c_3-8c_1c_2^{2}c_3 \\
     - & \, 17c_1^{2}c_3^{2}+4c_1c_2c_3^{2}-41c_1^{3}-45c_1^{2}c_2-13c_1c_2^{2}-4c_2^{3}+57c_1^{2}c_3 \\
     + & \, 38c_1c_2c_3+16c_2^{2}c_3-25c_1c_3^{2}-20c_2c_3^{2}+8c_3^{3}-32c_1^{2}-34c_1c_2 \\
      - & \, 15c_2^{2}+39c_1c_3+34c_2c_3-19c_3^{2}-14c_1-12c_2+13c_3-3. 
\end{align*}}
     

If $c_1+c_2+c_3+c_4+c_5+c_6=-1$, then $\det(I-\bigwedge^2A)=-1$, and 
{\allowdisplaybreaks %
\begin{align*}
\det(I-\bigwedge^3A)
= & \; -c_1^{7}-c_1^{6}-6c_1^{5}c_2+2c_1^{5}c_3+11c_1^{5}-4c_1^{4}c_2-9c_1^{3}c_2^{2}-8c_1^{4
     }c_3+10c_1^{3}c_2c_3 \\
     - & \, c_1^{3}c_3^{2}+22c_1^{4}+29c_1^{3}c_2+c_1^{2}c_2^{2}-4c_1c_2^{3}-37
     c_1^{3}c_3-18c_1^{2}c_2c_3+8c_1c_2^{2}c_3 \\
     + & \, 17c_1^{2}c_3^{2}-4c_1c_2c_3^{2}+5c_1^{3}+
     27c_1^{2}c_2+21c_1c_2^{2}+4c_2^{3}-23c_1^{2}c_3 \\
     - & \, 46c_1c_2c_3-16c_2^{2}c_3+25c_1c_3
     ^{2}+20c_2c_3^{2}-8c_3^{3}-8c_1^{2}-8c_1c_2 \\
     - & \, c_2^{2}+11c_1c_3+6c_2c_3-5c_3^{2}-2c_2+c_3+1. 
\end{align*}}
     
\end{small}

\end{document}